\documentclass{article}
\usepackage[a4paper, margin=3cm]{geometry}
\usepackage{amsmath,amssymb,amsthm}
\usepackage{mathtools}
\usepackage[backend=biber,style=alphabetic]{biblatex}
\usepackage{hyperref}
\usepackage{graphicx}
\usepackage{tikz-cd}
\usepackage{tikz}
\usepackage{mathrsfs}
\usepackage{xurl}

\newtheorem{Theorem}{Theorem}[section]
\newtheorem{Proposition}[Theorem]{Proposition}
\newtheorem{Corollary}[Theorem]{Corollary}
\newtheorem{Lemma}[Theorem]{Lemma}
\newtheorem{Definition}{Definition}[section]
\newtheorem{Remark}[Theorem]{Remark}
\newcommand{\va}{\boldsymbol{a}}
\newcommand{\vb}{\boldsymbol{b}}
\newcommand{\id}{\mathrm{id}}
\newcommand{\E}{\mathcal{E}}

\newcommand{\pE}{\mathcal{E}_{*}}
\newcommand{\fE}{{}^{i}\E}

\newcommand{\F}{\mathcal{F}}
\newcommand{\bF}{\boldsymbol{F}}
\newcommand{\pF}{\mathcal{F}_{*}}
\newcommand{\fF}{{}^{i}\F}
\newcommand{\bcF}{\boldsymbol{\mathcal{F}}}
\newcommand{\pS}{\mathcal{S}_{*}}
\newcommand{\Sb}{\mathcal{S}}
\newcommand{\G}{\mathcal{G}}

\newcommand{\pG}{\mathcal{G}_{*}}
\newcommand{\fG}{{}^{i}\G}

\newcommand{\T}{\mathcal{T}}
\newcommand{\U}{\mathcal{U}}
\newcommand{\ch}{\operatorname{ch}}

\newcommand{\tr}{\operatorname{Tr}}

\newcommand{\rank}{\operatorname{rank}}
\newcommand{\Gr}{\operatorname{Gr}}
\newcommand{\Hom}{\operatorname{Hom}}

\newcommand{\diff}{\mathop{}\!\mathrm{d}}  
\newcommand{\pdiff}{\mathop{}\!\mathrm{\partial}} 
\newcommand{\boplus}{\boldsymbol{\scalebox{1.3}{$\oplus$}}}

\newcommand{\wnablayl}{\widetilde{\nabla}^{1,0}}

\title{Kobayashi-Hitchin Correspondence for Saturated Reflexive Parabolic Sheaves on K\"ahler Manifolds}
\author{
  Tianshu Jiang\thanks{E-mail: jts2021@mail.ustc.edu.cn. School of Mathematical Sciences, University of Science and Technology of China, Hefei, 230026, P.R. China.} 
  \and 
  Jiayu Li\thanks{E-mail: jiayuli@ustc.edu.cn. School of Mathematical Sciences, University of Science and Technology of China, Hefei, 230026, and AMSS, CAS, Beijing, 100080, P.R. China.}
}

\begin{document}
\maketitle
\tableofcontents

\begin{abstract}
  In this paper, we study the Kobayashi-Hitchin correspondence in the setting of parabolic sheaves with a simple normal crossing divisor over a compact K\"ahler manifold using the method of Hermitian-Yang-Mills flow.
\end{abstract}
\section{Introduction}
\subsection{Background}

The concept of Kobayashi-Hitchin correspondence is developed from an insightful discovery of M. S. Narasimhan and C. S. Seshadri \cite{Na-Se} which reveals the correspondence between the irreducible unitary representations of the fundamental group of a compact Riemann surface and stable holomorphic vector bundles living on it with vanishing degree. With successive contributions from S. K. Donaldson~\cite{Do1,Do2}, S. Kobayashi~\cite{Ko1,Ko2}, L\"ubke~\cite{Lu}, Hitchin~\cite{Hc}, Uhlenbeck and Yau~\cite{Uh-Yau} etc., it takes into its modern shape which achieves the Trinity on a compact K\"ahler manifold $X$ as is illustrated by the following diagram.
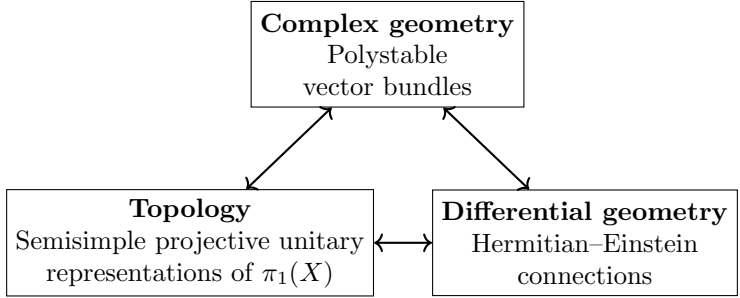
\begin{center}
\begin{tikzpicture}[
    every node/.style={draw, rectangle, align=center, minimum width=3.5cm, minimum height=1.4cm},
    arrow/.style={<->, thick}
]

\node (A) at (90:1cm) {\textbf{Complex geometry}\\Polystable\\vector bundles};
\node (B) at (210:3cm) {\textbf{Topology}\\Semisimple projective unitary\\representations of $\pi_1(X)$};
\node (C) at (330:3cm) {\textbf{Differential geometry}\\Hermitian--Einstein\\connections};

\draw[arrow] (A) -- (B);
\draw[arrow] (B) -- (C);
\draw[arrow] (C) -- (A);

\end{tikzpicture}
\end{center}
When we study the moduli space of the stable vector bundles over $X$, the correspondence enables the moduli space to interweave rich geometric structures from different perspectives.

Another direction of generalization is to consider the correspondence problem in the quasiprojective setting. It was pioneered by V. Metha and C. S. Seshadri~\cite{Me-Se} when they studied the unitary representations of the fundamental group of a punctured Riemann surface $X^{\circ}$, as a generalization of Narasimhan and Seshadri's~\cite{Na-Se} work. In their paper, they established the correspondence between the unitary representation of the fundamental group of $X^{\circ}$ and the concept of stable parabolic bundle over its completion $X$. Inspired by their work, the concept of parabolic vector bundles was gradually generalized to higher dimensional cases. For the purpose of illustration, let us consider the simplest case. 

Let $(X,\omega)$ be a compact K\"ahler manifold and $D$ be a smooth irreducible divisor. A parabolic bundle is a tuple $F_{*}=(F, \boldsymbol{F})$ consisting of an underlying holomorphic vector bundle $F\to X$ and a parabolic structure $\boldsymbol{F}$ along the divisor. The parabolic structure $\boldsymbol{F}$ is a descending filtration of holomorphic subbundles of $F_{|D}$:
\[
F_{|D} = F_0 \supsetneq F_1 \supsetneq F_2 \supsetneq \cdots \supsetneq F_{\ell} \supsetneq 0
\]
with weights $0\leq a_{i}<a_{i+1}<1$ assigned to $F_i$. 
We use the subscript ``*'' to emphasize the parabolic structure. And if we omit the ``*'' and just write $F$, we mean the underlying bundle of the parabolic bundle. As in the case of an ordinary vector bundle, associated to a parabolic bundle, we also have the parabolic versions of the Chern characters, the degrees, the slope stability with respect to $\omega$, etc. There are other characterizations of parabolic bundles provided, e.g., in~\cite{Mo2,Iy-Si,Bw}. More generally, the notions of parabolic sheaves with a general divisor are also available. It was firstly introduced by M. Maruyama and K. Yokogawa~\cite{Ma-Yo} to study the moduli problem. But in our paper, we follow the definitions of T. Mochizuki provided in~\cite[Chapter 3]{Mo1}. See also Section~\ref{2} for the details.

Parabolic bundles have attracted considerable attention since the 1990s. It is much more important and fascinating to study a parabolic bundle equipped with a Higgs field which has logarithmic singularities along the divisor. As it will lead to the nonabelian Hodge correspondence for quasiprojective smooth varieties pioneered by C. T. Simpson~\cite{Si2}. This subject tries to replace the objects in the above Trinity with semisimple local systems, harmonic bundles and stable parabolic Higgs bundles with trivial Chern classes over a quasiprojective variety. J. Jost and K. Zuo~\cite{Jo-Ka} built the bridge for the first two objects and T. Mochizuki~\cite{Mo1} built the bridge for the last two. O. Biquard~\cite{Bi} dealt with the case when the base space is a K\"ahler manifold with a smooth divisor. The problem is still open if the setting is a K\"ahler manifold with a simple normal crossing divisor. Nowadays, this subject is still active with researchers trying to generalize the settings to sheaves over spaces with certain singularities, e.g., Klt, Dlt varieties.
When there is no Higg field, this problem was also considered by, e.g.,~\cite{Li,Li-Na,St-Wr} in different settings. 

In particular, the second author~\cite{Li} proved that in the setting of a compact K\"ahler manifold $(X,\omega)$ with a simple normal crossing divisor $D$, a parabolic bundle $F_{*}$ is stable with respect $\omega$ if and only if there exists a Hermitian-Einstein (H-E) metric $H_{\delta}$ on $F_{|X\setminus D}$ with respect to a conical K\"ahler metrics $\omega_{\delta}$ (a metric with mild singularities near $D$), moreover $H_{\delta}$ is compatible with the parabolic structure. As a consequence, the author obtained the parabolic version of the Bogomolov-Gieseker inequality.

\subsection{Purpose}
It is worth noting that although the conical metric $\omega_{\delta}$ used by the second author are in the same cohomology class of $\omega$ in the sense of current, it is not very satisfactory to only get a H-E metric with respect to $\omega_{\delta}$ rather than $\omega$. The purpose of our paper is to generalize the previous results to the setting of semistable parabolic sheaves and in the mean time fix the above issue. We give a brief description of our main results in the following. 

Let $\pF$ be a saturated reflexive parabolic sheaf of rank $r$ associated to $(X,\omega, D)$. We put $\ch_{k}(\cdot)$ as the $k$-th Chern character operator of a parabolic sheaf. As usual, we define the degree $\deg_{\omega}(\pF)$ of a parabolic sheaf $\pF$ as $$\deg_{\omega}(\pF)\coloneqq \ch_{1}(\pF)\cdot \frac{[\omega]}{(n-1)!}$$ and the slope $\mu_{\omega}(\pF)$ as $$\mu_{\omega}(\pF)\coloneqq \frac{\deg_{\omega}(\pF)}{\rank(\pF)}.$$ 
We say that $\pF$ is $\mu_{\omega}$-stable (or simply $\omega$-stable) (resp. semistable) with respect to the K\"ahler metric $\omega$, if for any proper parabolic subsheaf $\pS$ we have 
\begin{equation*}
  \mu_{\omega}(\pS)<(resp. \leq)\, \mu_{\omega}(\pF).
\end{equation*} 
From Section~\ref{2}, we know that $\pF$ is a locally abelian parabolic bundle outside an analytic subset $W$ of codimension at least $3$. We put $X^{\circ}\coloneqq X-W$ and $F_{*}\coloneqq {\pF}_{|X^{\circ}}$. In the rest of the paper, we may sometimes simply call $F$ the regular part of $\F$ and $F_{|X^{\circ}\setminus D}$ the regular part of $\pF$.
\begin{Definition}\label{comatibleandadmissible}
  We say that an Hermitian metric $H$ on $F_{|X^{\circ}\setminus D}$ is compatible with the parabolic sheaf $\pF$ (or with its parabolic structure) if the following conditions are satisfied
  \begin{enumerate}
    \item \(\frac{\sqrt{-1}}{2\pi}\tr(F_{H})\) represents $\ch_{1}(\pF)$ where $F_{H}$ is the Chern curvature tensor. 
    \item For any proper parabolic subsheaves $\pS$ of $\pF$, let $H_{|\pS}$ be the induced metric on the regular part of $\pS$, we have \(\ch_{1}(\pS)\leq\frac{\sqrt{-1}}{2\pi}\tr(F_{H_{|\pS}})\) in the sense of current.
  \end{enumerate}
  And it is admissible if 
  \begin{itemize}
    \item $F_{H}\in L^{2}(H,\omega)$.
    \item $\Lambda F_{H}\in L^{\infty}(H)$ where $\Lambda F_{H}$ is the contraction of $F_{H}$ with $\omega$.
  \end{itemize}
\end{Definition}
\begin{Remark}
  The inequality in condition \(2\) could be strengthened to equality when $\pF$ is a parabolic bundle. But in the case of sheaves, this is the best we can expect.

  The concept of admissible metrics on reflexive sheaves was first introduced by S. Bando and Y. Siu in~\cite{B-S} where the Kobayashi-Hitchin problem of reflexive sheaves was concerned. 
\end{Remark}
We obtain the following theorems:
\begin{Theorem}[\textbf{Theorem} \ref{Stable}]
  A saturated reflexive parabolic sheaf $\pF$ over $(X,\omega,D)$ is $\mu_{\omega}$-polystable if and only if there exists an admissible Hermitian-Einstein metric with respect to $\omega$ on $F_{|X^{\circ}\setminus D}$ which is compatible with $\pF$. 
\end{Theorem}
\begin{Theorem}[\textbf{Theorem} \ref{Semistable}]
  A saturated reflexive parabolic sheaf $\pF$ over $(X,\omega,D)$ is $\mu_{\omega}$-semistable if and only if there exists a family of approximate Hermitian-Einstein metrics with respect to $\omega$ on $F_{|X^{\circ}\setminus D}$ all of which are compatible with $\pF$. 
\end{Theorem}
As a consequence, we obtain the parabolic version of the Bogomolov-Gieseker inequality. Let us denote \[\Delta(\pF)\coloneqq \frac{\ch_{1}(\pF)^2}{2\rank(\pF)}-\ch_{2}(\pF)\]
as the Bogomolov-Gieseker discriminant.
\begin{Corollary}[\textbf{Corollary} \ref{BGnormal}]
  If $\pF$ is $\mu_{\omega}$-semistable with respect to $\omega$, then $$\Delta(\pF)\cdot [\omega]^{n-2}\geq 0.$$ Moreover, if $\pF$ is polystable, then the equality holds if and only if ${\F}_{|X\setminus D}$ is a vector bundle which admits a projectively flat Hermitian-Einstein connection compatible with the parabolic structure.
\end{Corollary} 
More generally, 
\begin{Theorem}[\textbf{Theorem} \ref{nefbigBG}]
  Let $\pF$ be a saturated reflexive parabolic sheaf over a compact K\"ahler manifold $(X, \omega)$ which is semistable with respect to a nef and big class $[\eta]$. Then the Bogomolov-Gieseker inequality with respect to $[\eta]$ holds, i.e.,
  \begin{equation*}
    \Delta(\pF)\cdot [\eta]^{n-2}\geq 0.
  \end{equation*}
\end{Theorem}
\subsection{Outline of the paper}
\begin{itemize}
  \item $\S$\ref{2}: We introduce the abelian category whose objects are parabolic sheaves. Roughly speaking, a parabolic sheaf $\pF$ is a decreasing filtration of sheaves which only degenerate on a divisor \(D\). We discuss the notions of the morphisms, the Chern characters, the degrees, the stability condition, etc. We show that a saturated reflexive parabolic sheaf $\pF$ is a parabolic bundle outside an analytic set $W$ with codimension at least $3$. 
  \item $\S$\ref{3}: We use the Hironaka's theorem~\cite[p. 145, Corollary 2]{Hr} to resolve the singularities of $\pF$ by blowing up successively along the smooth centers above $W$. In summary, we will have a modification $\pi: \widetilde{X}\to X$. Let $\mathscr{E}$ be the exceptional divisor. To simplify notation, we use the same notation for objects pulled back by $\pi$. Then, over $\widetilde{X}$, $\F$ can be embedded into a locally free sheaf $E$ which is isomorphic to \(\F\) outside \(\mathscr{E}\). The embedding will give rise to an isomorphism between $\F$ and $E\otimes[-Q]$ where $Q$ is an effective divisor supported on $\mathscr{E}$. Then we show that by blowing up further and adding more components of $\mathscr{E}$ to $Q$, $E\otimes[-Q]$ will become a parabolic bundle $E\otimes[-Q]_{*}$ with the parabolic structure inherited from $\pF$. 
  \item $\S$\ref{4}: We first construct a smooth Hermitian metric $H_{0}$ on $E\otimes[-Q]_*$ which is adapted to its parabolic structure. Then we show that the curvature tensor of $H_{0}$ will give us the first and second Chern characters of $E\otimes[-Q]_*$ as in the case of ordinary vector bundles. We believe that the parabolic metric we have constructed will give us all of the parabolic Chern characters. In particular, it will induce a metric $\widehat{H}$ on $E$ by tensoring with a smooth metric of the line bundle $[Q]$ which in turn will induce a metric on the regular part of $\pF$ since $\widetilde{X}$ and $X$ are isomorphic outside $W$. As one can expect, the curvature tensor of $\widehat{H}$ will give us the first and second Chern characters of $\pF$. 
  \item $\S$\ref{5}: We recall the concept of the analytic stability of a vector bundle with respect to a Hermitian metric on the bundle and a K\"ahler form on the base manifold. More importantly, we show that the analytic stability of the Hermitian holomorphic bundle $(F_{|X^{\circ}\setminus D}, \widehat{H})$ is equivalent to the stability of the parabolic sheaf $\pF$.  
  \item $\S$\ref{6}: We investigate the long time existence of the Hermitian-Yang-Mills flow 
  \[
  \left\{
  \begin{aligned}
  &H^{-1}\cdot\frac{\diff H}{\diff t}=-2\left(\sqrt{-1}\Lambda_{\omega}F_{H}-\lambda\cdot\id_{F}\right) \\
  &\det(H)=\det(\widehat{H})\\
  &H(0)=\widehat{H}
  \end{aligned}
  \right.
  \]
  living on the bundle $F_{|X^{\circ}\setminus D}$ where $$\lambda\coloneqq \frac{2\pi\cdot \mu_{\omega}(\pF)}{\operatorname{Vol}(X,\omega)}.$$ 
  Once we have the long time existence of the flow, we will show that under the stability (resp. semistability) condition, the flow will converge to an admissible Hermitian-Einstein metric (resp. a family of admissible approximate Hermitian-Einstein metrics) on $F_{|X^{\circ}\setminus D}$ which is compatible with $\pF$. We will also prove the converse part where the admissibility of the Hermitian-Einstein metric is crucial.
  \item $\S$\ref{7}: We prove the Bogomolov-Gieseker type inequality for a semistable parabolic sheaf with respect to a big and nef class using Jordan-H\"older filtration.
\end{itemize}
\section*{Acknowledgement}
The first author is deeply grateful to Professor Takuro Mochizuki for his valuable discussions on the construction of the parabolic metric and the proof of its adaptedness.

The first author thanks Shiyu Zhang for organizing a meaningful seminar on this topic from which the author benefited greatly. The first author would also like to thank Changpeng Pan and Yang Zhou for useful discussions. 

\section{Elementary notions on parabolic sheaves}\label{2}
We review the elementary notions on parabolic sheaves. These concepts can be found, e.g., in \cite[Chapter 3]{Mo1}, \cite[Section 2]{Iy-Si}, \cite[Section 1]{Ma-Yo}. The definitions of parabolic sheaves introduced in these papers differ slightly but are all similar in spirit. Indeed, it was discussed in \cite{Iy-Si} that the diversified definitions appeared in the literature are equivalent when restricted to locally abelian parabolic bundles. In this paper, we basically follow the definitions introduced in \cite{Mo1}. 

Let $(X, D)$ be a pair of complex manifold $X$ and $D=\bigcup_{i\in I}D_{i}$ a simple normal crossing divisor. We impose a linear order on $I$. We introduce parabolic bundles associated to $(X, D)$ first and then do the generalizations to sheaves. If we say a generic property $P$ holds in codimension $k-1$, we mean that $P$ holds outside an analytic subset of codimension $k$. 

\subsection{Parabolic bundles}

We start with the notion of compatible filtrations. Fix a vector space $V$.
\begin{Definition}\label{def:filtration}
  A decreasing left continuous filtration of $V$ is a set of subspaces $F=\{F_{a}| a\in [0, 1[\}$ indexed by $a\in [0, 1[$ such that:
  \begin{enumerate}
    \item $F_{a}\supseteq F_{b}$ for $a\leq b$,
    \item $F_{0}=V$ and $F_{b}=0$ for $b$ sufficiently close to 1,
    \item If $\epsilon$ is sufficiently small, then $F_{a-\epsilon}=F_{a}$.
  \end{enumerate}
  All filtrations appear in this paper will be decreasing and left continuous hence we will simply call them filtrations in the rest of the paper.
\end{Definition}\label{def:flags}
   Given a tuple of filtrations $\boldsymbol{F}=\{{}^{i}F| i\in I\}$. For any $J\subset I$, $\boldsymbol{\eta}\in [0,1[^{J}$, we put as follows:
   \begin{equation*}
      {}^{J}F_{\boldsymbol{\eta}}=\bigcap_{j\in J}{}^{j}F_{\eta_{j}}
   \end{equation*}
\begin{Definition}\label{def:pcompatible}
  A tuple of filtrations $\boldsymbol{F}=\{{}^{i}F| i\in I\}$ is compatible if they admit a common splitting, i.e. a decomposition $V =\boldsymbol{\scalebox{1.3}{$\oplus$}}_{\boldsymbol{\eta} \in [0, 1[^{I}}U_{\boldsymbol{\eta}}$ such that:
  \begin{equation*}
    {}^{I}F_{\boldsymbol{\rho}}=\bigoplus_{\boldsymbol{\eta}\geq \boldsymbol{\rho}}U_{\boldsymbol{\eta}}
  \end{equation*}
\end{Definition}
For a coherent sheaf $\mathcal{V}$ over a complex manifold $M$, we define a decreasing left continuous filtration of subsheaves in the same manner. A filtration of subsheaves $\F=\{\F_{a}| a\in [0, 1[\}$ is called in the category of vector bundles, if all the subsheaves are locally free and $\Gr_{a}\coloneqq \F_{a}/\F_{>a}$ is locally free for all $a$.

\begin{Definition}\label{def:vcompatible}
  A tuple of vector bundle filtrations $\boldsymbol{F}=\{{}^{i}F| i\in I\}$ is compatible if:
  \begin{enumerate}
    \item For any $x\in X$, $\boldsymbol{F}_{x}$ is compatible in the sense of Definition \ref{def:pcompatible},
    \item For any $J\subset I$ and $\rho\in [0, 1[^{J}$, ${}^{J}F_{\rho}$ is a subbundle.
  \end{enumerate}
\end{Definition}
Let $F$ be a vector bundle over $X$. On each irreducible component $D_{i}$, one can specify a vector bundle filtration ${}^{i}F$ of $F|_{D_{i}}$. Hence one obtain a tuple of filtrations $\boldsymbol{F}=\{{}^{i}F| i\in I\}$. For any $J\subset I$, we put $D_{J}=\bigcap_{j\in J}D_{j}$. Then ${}^{J}\boldsymbol{F}$ is a tuple of filtrations of the vector bundle $F_{|D_{J}}$.
\begin{Definition}
  The tuple $(F, \bF)$ given as above is called a parabolic bundle, if for all $J\subset I$, restricting to any irreducible component of $D_{J}$, ${}^{J}\boldsymbol{F}$ is compatible in the sense of Definition \ref{def:vcompatible}. $\bF$ is called the parabolic structure on $F$.  
\end{Definition}
\begin{Definition}
  A parabolic bundle $(F, \bF)$ is called locally abelian, if for each $x\in X$, there is a neighborhood $U_{x}$ of $x$ such that $(F, \bF)_{|U_{x}}$ is isomorphic to a direct sum of parabolic line bundles in the category of parabolic sheaves \textnormal{(see the definition in the next subsection)}.
\end{Definition}
N. Borne \cite{Bo} shows
\begin{Theorem}\label{Locally Abelian}
  Given a parabolic bundle $(F, \bF)$ over $(X, D)$. If all of the intersections of the subsheaves $\fF_{a}$ defined by the exact sequence $$
  0\to \fF_{a} \to F \to F_{|D_{i}}/{}^{i}F_{a} \to 0
$$
are locally free, then $(F, \bF)$ is locally abelian.
\end{Theorem}
\begin{Remark}
  If the base space is a smooth algebraic variety, Borne's theorem is stronger. It says that the parabolic bundle is isomorphic to a direct sum of parabolic line bundles in any Zariski neighborhood.
\end{Remark}
\subsection{Parabolic sheaves}
Let $\F$ be a torsion-free coherent sheaf of $\mathcal{O}_{X}$-modules.
\begin{Definition}
  A parabolic structure on $\F$ is a tuple $\bcF=\{ \fF \mid i \in I \}$ of decreasing left continuous filtrations $\fF_{a} $ indexed by $a \in [0, 1[$ with finite length $l_{i}$ such that $\fF_{a} \supset \F(-D_{i})  $ for any $a \in [0, 1[$ and $\fF_{a} = \F(-D_{i})$ if $a$ is sufficiently large.
\end{Definition}
For the filtration $\fF$ over $D_{i}$, let ${}^{i}a=\{{}^{i}a_{0}, {}^{i}a_{1},\cdots, {}^{i}a_{l_{i}}\}$ be the increasing sequence of indexes such that ${}^{i}\Gr_{{}^{i}a_{k}}\pF\neq 0$. This is often called the weights of the parabolic structure on $D_{i}$. We call $\F$ the top flag and $\F(-D_{i})$ the handle in the filtration $\fF$.
\begin{Definition}
  A parabolic sheaf is a tuple $\pF=\left(\F, \bcF\right)$ consisting of a underlying sheaf $\F$ and a parabolic structure $\bcF$ on $\F$.
\end{Definition}
\begin{Definition}
  A parabolic subsheaf of $\pF=\left(\F, \bcF\right)$ is defined as a quotient torsion-free subsheaf $\E$ with the naturally induced parabolic structure. 
\end{Definition}

\begin{Definition}
  A morphism $f: \pF\to \pG$ is defined as a morphism of sheaves $f: \F\to \G$ such that $f(\fF_{a})\subseteq \fG_{a}$ for all $a\in [0, 1[$.
\end{Definition}
We denote the set of the morphisms between two parabolic sheaves $\pF$ and $\pG$ as $\Hom(\pF, \pG)$.

\begin{Definition}
  A complex $$\cdots \to \pE\to \pF \to \pG\to \cdots$$ is exact at $\pF$, if and only if $$\cdots \to \fE_{a}\to \fF_{a}\to \fG_{a}\to \cdots$$ is exact at $\fF_{a}$ for all $i$ and $a$.
\end{Definition}
A parabolic bundle gives rise to a parabolic sheaf in the following way. Given a parabolic bundle $F_{*}$, one define the parabolic structure $\fF_{a}$ by the exact sequence $$
  0\to \fF_{a} \to F \to F_{|D_{i}}/{}^{i}F_{a} \to 0.
$$
Conversely, how far away is a parabolic sheaf from a parabolic bundle? The following lemmas give a partial answer.

We say that a parabolic sheaf $\pF$ is reflexive if the underlying sheaf $\F$ is reflexive. And $\pF$ is called saturated if $\F/\fF_{a}$ are torsion-free $\mathcal{O}_{D_{i}}$-modules for any $i$ and $a$. We put $${\fF_{a}}_{|D_{J}}\coloneqq \Im ({}^{J}m^{*}(\fF_{a})\to {}^{J}m^{*}(\F))$$ where ${}^{J}m$ is the embedding of $D_{J}$.

\begin{Lemma}\label{lem:reflexive}
  All the intersections of subsheaves belonging to the parabolic structure of a saturated reflexive parabolic sheaf are reflexive.
  \begin{proof}
    T. Moichizuki~\cite[Proposition 3.1.2]{Mo1} shows that all the subsheaves $\fF_{a}$ are reflexive. Then the lemma follow from the fact that the intersection of two reflexive subsheaves of a reflexive sheaf is reflexive. 
  \end{proof}
\end{Lemma}
\begin{Lemma}\label{lem:abelian}
  A saturated reflexive parabolic sheaf is a locally abelian parabolic bundle in codimension $2$.
  \begin{proof}
    Denote the sheaf by $\pF$. Since $\F$ is reflexive, $\F$ is locally free outside an analytic set $Z_{0}$ with codimension $3$. On each $D_{i}$, since $\F/\fF_{a}$ are torsion-free $\mathcal{O}_{D_{i}}$-modules, they are locally free outside an analytic subset $Z_{i}$ of codimension $2$ in $D_{i}$. Hence outside $\bigcup_{\{0,I\}}Z_{i}$, ${}^{i}\F_{|D_{i}}$ is a vector bundle filtration for all $i$. It is easy to check that for $i\neq j$, the vector bundle filtrations ${}^{i}\F_{|D_{ij}}$ and ${}^{j}\F_{|D_{ij}}$ are compatible outside an analytic subset $Z_{ij}$ of codimension $2$ in $D_{ij}$. Hence if we further delete the union of $Z_{ij}$ and $D_{J}$ with $|J|\geq 3$, $\pF$ will be a parabolic bundle. By the above lemma, all of the intersections of $\fF_{a}$ are reflexive, hence they are locally free outside an analytic subset $Z'$ of codimension $3$. Deleting $Z'$ makes $\pF$ locally abelian.  
  \end{proof}
\end{Lemma}
\begin{Lemma}
  A parabolic sheaf admits a unique reflexive saturation which is isomorphic to the former in codimension 1, i.e. given a parabolic sheaf $\F_{*}$, there exists a unique saturated reflexive parabolic sheaf $\F_{*}'$ and a monomorphism $m: \F_{*}\rightarrow \F_{*}'$ which is an isomorphism in codimension 1.
  \begin{proof}
    Let $\F'$ be the double dual of $\F$. They are isomorphic outside an analytic subset $Z$ of codimension 2. Since $\F'$ is normal, we define a parabolic structure $\boldsymbol{\F}^{1}$ as the unique extension of $\boldsymbol{\F}_{|X-Z}$ in $\F'$. The subsheaves in $\boldsymbol{\F}^{1}$ are coherent due to Theorem 2.2 of \cite{Si-Tr}. Moreover, the top flag of ${}^{i}\F$ is $\F'$ and the handle is $\F'(-D_{i})$. Hence $\F'/{}^{i}\F_{a}^{1}$ can be regarded as a sheaf of $\mathcal{O}_{D_{i}}$-modules. Finally we construct ${}^{i}\F_{a}'$ as the inverse image of the $\mathcal{O}_{D_{i}}$-torsion subsheaf of $\F'/{}^{i}\F_{a}^{1}$ under the quotient map. One readily verifies that $\F'_{*}=(\F', \boldsymbol{\F}')$ serves as an candidate. The uniqueness follows immediately from Lemma \ref{lem:reflexive} and the fact that a reflexive sheaf is normal.
  \end{proof}
\end{Lemma}

\subsection{Parabolic Chern character}
From now on, we assume that $X$ is a compact K\"ahler manifold. 
\begin{Definition}
  The parabolic Chern character $\ch(\pF)$ of a parabolic sheaf $\pF$ is given by the formula 
  \begin{equation}\label{integralchern}
    \ch(\pF)=\frac{\ch(D)\int_{a_{1}=0}^{1}\cdots\int_{a_{k}=0}^{1}e^{\sum_{i=1}^{k}a_{i}\cdot[D_{i}]}\cdot\ch({}^{I}\F_{\boldsymbol{a}})}{\int_{a_{1}=0}^{1}\cdots\int_{a_{k}=0}^{1}e^{\sum_{i=1}^{k}a_{i}\cdot[D_{i}]}}
\end{equation}
where $k$ is the cardinality of $I$ and $\boldsymbol{a}\in [0, 1[^{k}$.
\end{Definition}
\begin{Remark}
  In this paper, we use the Chern character functor for an ordinary coherent analytic sheaf over a compact complex manifold provided in~\cite{Kf} where they generalized the intersection theory (more precisely, the GRR theorem) in the smooth algebraic setting to the smooth complex analytic setting. The above definition of the parabolic Chern character is based on their definition for an ordinary sheaf. However it is by no means standard. The above formula was originally obtained by Ivy and Simpon in \cite{Iy-Si} for locally abelian parabolic bundles. They proved the equivalence between the category of locally abelian parabolic bundles with rational weights over a pair $(X, D)$ of a smooth scheme and a simple normal crossing divisor and the category of vector bundles over an associated Deligne-Mumford stack $Z$. Then they used the equivalence to define the parabolic Chern character of a locally abelian parabolic bundle and calculated the explicit formula as above. But for more general parabolic sheaves or even parabolic bundles, there is no hint for the above formula to hold. We use it as a definition for the following reasons.
\end{Remark}
Firstly, such a definition of Chern character has good functorial properties. For instance, given a short exact sequcence of parabolic sheaves $$
0\to \pE\to \pF\to \pG\to 0,
$$
we have $\ch(\pF)=\ch(\pE)+\ch(\pG)$. Indeed this is all we need for the purpose of proving the Bogomolov-Gieseker inequality in Section~\ref{7}.

More importantly, it was calulated by Taher in \cite{Ta} that for a locally abelian parabolic bundle in codimension 1 (resp. 2), we have following more concrete formulas for the first (resp. second) Chern character. 

Suppose $\pF$ is a locally abelian parabolic bundle in codimension $2$, e.g. a saturated reflexive parabolic sheaf. We put ${}^{i}\Gr\pF\coloneqq \boldsymbol{\scalebox{1.3}{$\oplus$}}_{a\in [0, 1[}({}^{i}\Gr_{a}\pF \coloneqq \fF_{a}/\fF_{>a})$ as a graded sheaf of $O_{D_{i}}$-modules. On any irrducible component $P$ of $D_{i}\cap D_{j}$, we put $${}^{P}\Gr_{(a_{i},a_{j})}\pF \coloneqq {}^{P}F_{a_{i},a_{j}}/\sum_{x>a_{i},y>a_{j}}{}^{P}F_{x,y}.$$ Then the formula reads as follows:

\begin{align}\label{chernform0}
  \ch_1(\pF)&=\ch_{1}(\F)+\sum_{i\in I}\sum_{a\in [0, 1[}a\cdot\rank_{D_{i}}({}^{i}\Gr_{a}\pF)\cdot[D_{i}]\\
  \ch_{2}(\pF)&=\ch_{2}(\F)+\sum_{i\in I}\sum_{a\in [0, 1[}a\cdot m_{i\mkern1mu*}\left(\ch_{1}({}^{i}\Gr_{a}\pF)\right)\nonumber \\ 
  &\quad +\frac{1}{2}\sum_{i\in I}\sum_{a\in [0, 1[}a^{2}\cdot\rank_{D_{i}}({}^{i}\Gr_{a}\pF)\cdot[D_{i}]^{2}\nonumber \\
  &\quad +\sum_{\substack{i < j\\(i, j)\in I^{2}}}\sum_{\substack{P\in Irr(D_{i}\cap D_{j}\\(a_{i},a_{j})\in [0, 1[^{2}}}a_{i}\cdot a_{j}\cdot\rank_{P}{}^{P}\Gr_{(a_{i},a_{j})}\cdot[P].
\end{align}

\begin{Remark}
  The above formula of the fisrt Chern character was firstly introduced by V. Metha and C.S.Seshadri \cite{Me-Se} on algebraic curves and then generalized to parabolic sheaves in higher dimensions by M. Maruyama and K. Yokogawa \cite{Ma-Yo}. The formula of the second Chern character for a locally abelian parabolic bundles was obtained by the second author~\cite{Li} using the curvature tensor of some adapted metric which degenerates on the divisor. And then T.Mochizuki \cite{Mo1} used it as a definition for saturated reflexive parabolic sheaves. After Iyer and Simpson \cite{Iy-Si} gave a general definition of parabolic Chern character for locally abelian parabolic bundles, Taher showed in \cite{Ta} that the new general definition agrees with the classical ones. 
\end{Remark}

Since in this paper, we only study the reflexive saturated parabolic sheaves, the Chern character functor given in the very beginning of this subsection serves as a good candidate. Next, we compare the Bogomolov-Gieseker discriminant of a parabolic sheaf with that of its reflexive saturation in the following lemma.
\begin{Lemma}\label{BGreflexivication}
  Let $\pF'$ be the reflexive saturation of a parabolic sheaf $\pF$. Then $\Delta(\pF')\geq \Delta(\pF)$.
  \begin{proof}
    On the one hand, since they are isomorphic in codimension 1, $\mathcal{T}_{\boldsymbol{a}}={}^{I}\F_{\boldsymbol{a}}'/{}^{I}\F_{\boldsymbol{a}}$ is a torsion sheaf supported on an analytic subset of codimension 2. By \cite[Theorem 1.1]{Kf}, $\ch_{1}(\mathcal{T}_{\boldsymbol{a}})$ vanishes and $\ch_{2}(\mathcal{T}_{\boldsymbol{a}})$ is a current represented by a nonnegative linear sum of irreducible components of $\text{supp}\mathcal{T}_{\boldsymbol{a}}$ with codimension $2$. On the other hand, by equation \eqref{integralchern}, we have $$
    \ch(\pF')-\ch(\pF)=\frac{\ch(D)\int_{a_{1}=0}^{1}\cdots\int_{a_{k}=0}^{1}e^{\sum_{i=1}^{k}a_{i}\cdot[D_{i}]}\cdot\ch(\T_{\boldsymbol{a}})}{\int_{a_{1}=0}^{1}\cdots\int_{a_{k}=0}^{1}e^{\sum_{i=1}^{k}a_{i}\cdot[D_{i}]}}.
    $$
    Then it is easy to observe that $\ch_{1}(\pF')=\ch_{1}(\pF)$ and $\ch_{2}(\pF')-\ch_{2}(\pF)$ is a current represented by a nonnegative linear sum of irreducible components of $\text{supp}\mathcal{T}_{\boldsymbol{a}}$ with codimension $2$. Hence $\Delta(\pF')\geq \Delta(\pF)$.
  \end{proof}
\end{Lemma}
The notions of stability, polystability and semistability for parabolic sheaves are defined in the same way as for an ordinary sheaf. The existence and uniqueness of Harder-Narasimhan filtration and Jordan-H\"older filtration can be proved in a tautological way as for an ordinary sheaf. Also as in the case of an ordinary sheaf, 
\begin{Lemma}\label{stableofreflexivacation}
  If a parabolic sheaf $\pF$ is stable with respect to a class $[\eta]\in H^{2}(X, \mathbb{R})$, then its reflexive saturation $\pF'$ is also $\eta$-stable. 
\end{Lemma}

\section{Resolution of singularities}\label{3}
Roughly speaking, we blow up along the singular locus $W$ of the parabolic sheaf mentioned in Lemma~\ref{lem:abelian}, and use Hironaka's theorem \cite[Corollary 2]{Hr} to resolve the singularities of the parabolic sheaf. We proceed as follows.

Let $\F^\vee$ be the dual of $\F$. Then we have a tuple of locally free resolutions of $\F^\vee_{U_{\alpha}}$ associated to an open cover $\{U_{\alpha}\}$ of $X$. On any overlap $U_{\alpha}\cap U_{\beta}$, we have a morphism of complexes 
\[
  \begin{tikzcd}
    \cdots \arrow[r] &{}^{\beta}\E_{1}^\vee \arrow[r, "\beta_{1}^\vee"] & {}^{\beta}\E_{0}^\vee \arrow[r, "\beta_{0}\vee" ] \arrow[d, "\phi_{\alpha\beta}^\vee"]& \F_{|U_{\alpha}\cap U_{\beta}}^\vee \arrow[r]  \arrow[d, "id"]& 0\\
    \cdots \arrow[r] &{}^{\alpha}\E_{1}^\vee \arrow[r, "\alpha_{1}^\vee"] & {}^{\alpha}\E_{0}^\vee \arrow[r, "\alpha_{0}\vee" ] & \F_{|U_{\alpha}\cap U_{\alpha}}^\vee \arrow[r] & 0.
    \end{tikzcd}
\]
Dualizing, we get:
\[
  \begin{tikzcd}
    0 \arrow[r] & \F_{|U_{\alpha}\cap U_{\beta}} \arrow[r, "\alpha_{0}"] \arrow[d, "id"] & {}^{\alpha}\E_{0} \arrow[r, "\alpha_{1}"] \arrow[d, "\phi_{\alpha\beta}"] & {}^{\alpha}\E_{1}  \\
    0 \arrow[r] & \F_{|U_{\alpha}\cap U_{\beta}} \arrow[r, "\beta_{0}"] & {}^{\beta}\E_{0} \arrow[r, "\beta_{1}"]  & {}^{\beta}\E_{1}.
    \end{tikzcd}
\]
By \cite[Corollary 2]{Hr}, we can find a modification $\pi:\widetilde{X} \to X$ which is obtained by successively blowing up along the smooth centers supported on the singular locus of $\F$, such that for any $\alpha$, $E_{\alpha}\coloneqq\ker(\pi^{*}\alpha_{1})$ is locally free. Since $\pi^{*}\phi_{\alpha\beta}: E_{\alpha} \to E_{\beta}$ is an identity in codimension $0$, we can glue up $E_{\alpha}$'s to be a global bundle $E$ on $\widetilde{X}$. At this stage, we have $\pi^{*}(\F)/\operatorname{Tor}{\pi^{*}(\F)}$ naturally inject into $E$ as a subsheaf with the same rank of $E$. Note that if we keep blowing up, the former good properties will be preserved. To save notation, if we blow up again, we will still denote the modification as $\pi:\widetilde{X} \to X$ and the pullback of $E$ as $E$. We also follow the convention that we always modulo the torsion of the pullback of a sheaf and use the same notation for the pullbacked sheaf. Moreover, we will not change the notation for the inverse images of the irreducible components of the exceptional divisors generated on the process of the successive blow-ups. The proper transformation of $D$ ($D_{i}$) will always be denoted as $D^*$ ($D_i^*$) and the exceptional divisor will be denoted as $\mathscr{E}$. Let us move on. 

By \cite[Corollary 2]{Hr} again and repeat the above process, there will be an effective divisor $P$ supported on $\mathscr{E}$ such that the canonical map from $\pi^{*}(\F)$ to $E$ maps $\pi^{*}(\F)$ isomorphically onto $E\otimes[-P]$. Next we blow up the irreducible components of the intersections of more than $2$ divisors. Then on $\widetilde{X}$, $D_i^*\cap D_j^*\cap D_k^*=\varnothing$ if $i\neq j\neq k$. Blow up again to make ${}^{J}\F_{{}^{J}a}$ locally free for all $J\subset I$ and ${}^{J}a\in [0,1[^{J}$. Notice that the good properties we have obtained before will be preserved under further blow-ups. 

Now we are at the stage to deal with the singularities of the parabolic structure. 
Let $r$ be the rank of $E$. We put ${}^{i}r\coloneqq\{{}^{i}r_{0},\cdots, {}^{i}r_{l_{i}}\}$ where ${}^{i}r_{k}$ is the rank of ${}^{i}\Gr_{{}^{i}a_{k}}F_{*}$. Set ${}^{i}r_{-1}=0$. We will show in the following lemma that by dealing with the singularities of the parabolic structure carefully, for each $^i\F_{{}^{i}a_{k}}$ we can find an effective divisor ${}^iP_k$ supported on $\mathscr{E}$. We tensor it with $[-{}^iP_k]$ to get $^iE_{{}^{i}a_{k}}\coloneqq ^i\F_{{}^{i}a_{k}}\otimes [-{}^iP_k]$. We may also find a $P_{0}$ to get $E\otimes[-P_0]$. We will prove that $^iE_{{}^{i}a_{k}}$ forms a parabolic filtration of $E\otimes[-P_{0}]$ over $D_i^*$ as $k$ increases. Denote the filtration by ${}^i\boldsymbol{E}$. Then we have:
\begin{Proposition}\label{Good frame}
  The tuple $(E\otimes[-P_0],({}^i\boldsymbol{E})_i)$ is a locally abelian parabolic bundle associated to $(\widetilde{X},D^*)$.
  
  \begin{proof}
    We put $Q_0\coloneqq P$. If $D_{i}^*$ doesn't intersect with any other $D_{j}^*$. We have 
     \[
   \begin{tikzcd}
     {\fF_{{}^{i}a_{1}}}_{|D_{i}^*} \arrow[r,hookrightarrow] & \F_{|D_{i}^*} \arrow[r,"\cong"] & E\otimes[-Q_0]_{|D_{i}^*}
   \end{tikzcd}
 \]
 as a torsion-free subsheaf of $E\otimes[-Q_0]_{|D_{i}^*}$. Blow up to make it isomorphic to a subbundle of $E\otimes[-Q_0-Q_{1}]_{|D_{i}^*}$. And we have a sequence of induced morphisms
 \[
   \begin{tikzcd}
    {\fF_{{}^{i}a_{2}}}_{|D_{i}^*} \arrow[r,hookrightarrow]&{\fF_{{}^{i}a_{1}}}_{|D_{i}^*} \arrow[r,hookrightarrow] & \F\otimes[-Q_{1}]_{|D_{i}^*} \arrow[r,"\cong"] & E\otimes[-Q_0-Q_{1}]_{|D_{i}^*}
   \end{tikzcd}
 \]
 which implies that ${\fF_{{}^{i}a_{2}}}_{|D_{i}^*}$ is a torsion-free subsheaf of $E\otimes[-Q_{0}-Q_{1}]_{|D_{i}^*}$. Blow up again, to make it isomorphic to a subbundle of $E\otimes[-Q_0-Q_{1}-Q_{2}]_{|D_{i}^*}$. And we have a sequence of induced morphisms 
 \[
   \begin{tikzcd}[column sep=small]
    {\fF_{{}^{i}a_{2}}}_{|D_{i}^*} \arrow[r,hookrightarrow]&{\fF_{{}^{i}a_{1}}\otimes[-Q_{2}]}_{|D_{i}^*} \arrow[r,hookrightarrow] & \F\otimes[-Q_{1}-Q_{2}]_{|D_{i}^*} \arrow[r,"\cong"] & E\otimes[-Q_{0}-Q_{1}-Q_{2}]_{|D_{i}^*}.
   \end{tikzcd}
 \]
So on and so forth, until we exhaust all the filtrations over $D_{i}^*$. We put $E\otimes[-P_0]\coloneqq E\otimes[-Q_{0}-\cdots-Q_{l_{i}}]\cong\F\otimes[-Q_{1}-\cdots-Q_{l_{i}}]$, $^iE_{{}^{i}a_{k}}\coloneqq ^i\F_{{}^{i}a_{k}}\otimes [-{}^iP_k]\coloneqq \fF_{a_{k}}\otimes[-Q_{k+1}-\cdots-Q_{l_{i}}]$. Since all of them are locally free and they are descending subbundles of $E\otimes[-P_0]_{D_{i}^*}$ restricting to $D_{i}^*$, then by Theorem \ref{Locally Abelian}, it is a locally abelian parabolic bundle over the pair $(\widetilde{X}, D_{i}^*)$. 

If $D_{i}^*$ intersects with some $D_{j}^*$, for the purpose of illustrasion, we assume the simplest case that the length of the filtrations $\fF$ and ${}^{j}\F$ are $1$. We put $\G_{10}\coloneqq \fF_{{}^{i}a_{1}}$, $\G_{01}\coloneqq {}^j\F_{{}^{j}a_{1}}$ and $\G_{11}\coloneqq \G_{10}\cap \G_{01}$. As above we have 
\[
   \begin{tikzcd}
    {\G_{10}}_{|D_{i}^*} \arrow[r,hookrightarrow] & \F\otimes[-{}^iQ_{1}]_{|D_{i}^*} \arrow[r,"\cong"] & E\otimes[-Q_{0}-^{i}Q_{1}]_{|D_{i}^*}.
   \end{tikzcd}
 \]
 And we have the naturally induced morphism 
\[ 
  \begin{tikzcd}
    {\G_{01}}\otimes[-{}^{i}Q_{1}]_{|D_{j}^*} \arrow[r, hook] & \F\otimes[-{}^{i}Q_{1}]_{|D_{j}^*} \arrow[r,"\cong"] & E\otimes[-Q_{0}-{}^{i}Q_{1}]_{|D_{j}^*}.
    \end{tikzcd}
 \] 
 Blow up to make ${\G_{01}}\otimes[-Q_{1}]_{|D_{j}^*}$ being isormorphic to a subbundle of $E\otimes[-Q_{0}-{}^{i}Q_{1}-{}^{j}Q_{1}]_{|D_{i}^*}$, and we have the induced morphisms
 \[ 
  \begin{tikzcd}
    {\G_{10}}\otimes[-{}^{j}Q_{1}]_{|D_{i}^*} \arrow[r, hook] 
    & \F\otimes[-{}^{i}Q_{1}-{}^{j}Q_{1}]_{|D_{i}^*} \arrow[r,"\cong"] & E\otimes[-Q_{0}-{}^{i}Q_{1}-{}^{j}Q_{1}]_{|D_{i}^*},\\
    {\G_{01}}\otimes[-{}^{i}Q_{1}]_{|D_{j}^*}\arrow[r, hook]& \F\otimes[-{}^{i}Q_{1}-{}^{j}Q_{1}]_{|D_{j}^*} \arrow[r,"\cong"] & E\otimes[-Q_{0}-{}^{i}Q_{1}-{}^{j}Q_{1}]_{|D_{j}^*}.
    \end{tikzcd}
 \]
 Then ${\G_{11}}_{|D_{ij}^*}$ is a torsion-free subsheaf of $E\otimes[-Q_{0}-{}^{i}Q_{1}-{}^{j}Q_{1}]_{|D_{ij}^*}$ and we can blow up to make it being isomorphic to a subbundle of $E\otimes[-Q_{0}-{}^{i}Q_{1}-{}^{j}Q_{1}-Q_{11}]_{|D_{ij}^*}$. And we have the induced morphisms  
 \[
  \begin{tikzcd}[column sep=small, row sep=small]
    & {\G_{10}}\otimes[-{}^{j}Q_{1}-Q_{11}]_{|D_{ij}^*} \arrow[dr, hook] \\
    {\G_{11}}_{|D_{ij}^*} \arrow[ur, hook] \arrow[dr, hook] 
    & & \F\otimes[-{}^{i}Q_{1}-{}^{j}Q_{1}-Q_{11}]_{|D_{ij}^*} \arrow[dd, "\cong"] \\
    & {\G_{01}}\otimes[-{}^{i}Q_{1}-Q_{11}]_{|D_{ij}^*} \arrow[ur, hook] \\
    & & E\otimes[-Q_{0}-{}^{i}Q_{1}-{}^{j}Q_{1}-Q_{11}]_{|D_{ij}^*}.
  \end{tikzcd}
\]
 We put $E\otimes[-P_0]\coloneqq E\otimes[-Q_{0}-{}^{i}Q_{1}-{}^{j}Q_{1}-Q_{11}]\cong\F\otimes[-{}^{i}Q_{1}-{}^{j}Q_{1}-Q_{11}]$, ${}^{i}E_{1}\coloneqq {\G_{10}}\otimes[-{}^{j}Q_{1}-Q_{11}]$ and ${}^{j}E_{1}\coloneqq {\G_{01}}\otimes[-{}^{i}Q_{1}-Q_{11}]$. Since all of them are locally free, and ${{}^iE_{1}}_{|D_{i}}$, ${{}^jE_{1}}_{|D_{j}}$ and ${{}^iE_{1}\cap{}^jE_{1}}_{|D_{ij}}$ are a subbundles. Hence by Theorem \ref{Locally Abelian}, $(E\otimes[-P_0], \{{i}F_{1},{j}F_{1}\})$ is locally abelian. 

 Since three divisors cannot intersect simultaneously, we can repeat the above processes to deal with all the singularities of the parabolic structure. The proof is completed.
  \end{proof}
\end{Proposition} 
Without loss of generality, we may assume that $\mathscr{E}\cup D$ is a simple normal crossing divisor. In the rest of the paper, we denote the locally abelian parabolic bundle $(E\otimes[-P_0], \{{}^i\boldsymbol{E}\}_{i\in I})$ as $E_{*}'$. By the construction of $E_{*}'$, we have the following lemma.
\begin{Lemma}\label{cherncharacterpreserve}
  \begin{align*}
    \ch_{1}(\pF)\cdot[\omega^{n-1}]&=\ch_{1}(E_*')\cdot\pi^*[\omega^{n-1}]\\
    \ch_{1}(\pF)^2\cdot[\omega^{n-2}]&=\ch_{1}(E_*')^2\cdot\pi^*[\omega^{n-2}]\\
    \ch_{2}(\pF)\cdot[\omega^{n-2}]&=\ch_{2}(E_*')\cdot\pi^*[\omega^{n-2}]\\
  \end{align*}
  \begin{proof}
    By equation~(\ref{integralchern}), we have 
    \begin{align*}
      &\pi^*\ch_{1}(\pF)-\ch_{1}(E_*')=\{\text{a divisor supported on $\mathscr{E}\cup (D-D^*)$}\}\\
      &\pi^*\ch_{2}(\pF)-\ch_{2}(E_*')=\{\text{linear sum of $D_i^*\cdot Q_i$s with $Q_i$s supported on $\mathscr{E}\cup (D-D^*)$}\}
    \end{align*}
    Since $W$ has codimension at least $3$, the lemma follows.
  \end{proof}
\end{Lemma}
\section{Constructions and analyses of metrics}\label{4}
In this section, we first construct a family of conical K\"ahler metrics $\omega_{\delta}$ on $X\setminus D$ and a family of conical K\"ahler metrics $\omega_{\epsilon\delta}$ on $\widetilde{X}\setminus D$. 
Then we construct a Hermitian metric $\widehat{H}$ on the regular part of $\pF$ which is compatible with the parabolic structure in the sense of Definition~\ref{comatibleandadmissible}. The metrics will be used in Section~\ref{6} when we study the Hermitian-Yang-Mills flow.

Let $h_{i}$ be a Hermitian metric of $\mathcal{O}_{X}(D_i)$. Let $\sigma_i$ denote the canonical section of $\mathcal{O}_{X}(D_i)$. The norm of $\sigma_i$ with respect to $h_i$ is denoted by $|\sigma_i|$. The canonical section of $\mathcal{O}_{X}(D)$ is denoted by $\sigma$ and the norm of $\sigma$ is denoted by $|\sigma|$. We assume that $|\sigma|<1$. Also as in the last section, we don't change the notation for geometric objects under the pullback of $\pi:\widetilde{X}\to X$.
\subsection{Conical K\"ahler metrics}
We construct a family of regularized metrics on $X$ whose limit gives us a conical K\"ahler metric $\omega_\delta$ on $X\setminus D$. This kind of regularization trick is commonly used to deal with the geometric problems on conical K\"ahler manifolds, cf.~\cite{Ga-Hi-Pn-Mi}. We introduce for any $0\leq\nu$ and $0<\delta<\frac{1}{2}$ the function 
\begin{equation*}
  \chi_{\delta\nu}(t)\coloneqq\delta\int_{0}^{t}\frac{(\nu^2+s)^\delta-\nu^{2\delta}}{s}\diff s.
\end{equation*}
 We define the following, for some positive number $C>0$, 
\begin{equation*}
  \omega_{\delta\nu}\coloneqq \omega + C\cdot\sum_{i\in I}\sqrt{-1}\cdot\partial\overline{\pdiff}\chi_{\delta\nu}(|\sigma_i|^2).
\end{equation*}
\begin{Proposition}
  For $C$ sufficiently small, if $0<\nu$, $\omega_{\delta\nu}$ is a K\"ahler metric on $X$ for any $\delta$, and if $\nu=0$, $\omega_\delta\coloneqq\omega_{\delta0}$ is a K\"ahler metric on $X\setminus D$ for any $\delta$.
  \begin{proof}
    Direct calculations show that 
    \begin{equation*}
      \sqrt{-1}\cdot\partial\overline{\partial}\chi_{\delta\nu}(|\sigma_i|^2)=\sqrt{-1}\cdot\delta^2\cdot(\nu^2+|\sigma_i|^2)^{\delta}\cdot\frac{\langle\partial_{h_i}\sigma_i, \partial_{h_i}\sigma_i\rangle}{\nu^2+|\sigma_i|^2}-\sqrt{-1}\cdot\delta\cdot((\nu^2+|\sigma_i|^2)^{\delta}-\nu^{2\delta})\cdot F_{h_i}.
    \end{equation*}
    The proposition follows from the uniform boundedness of the quantities: $\frac{\langle\partial_{h_i}\sigma_i, \partial_{h_i}\sigma_i\rangle}{\nu^2+|\sigma_i|^2}$, $F_{h_i}$, $\delta\cdot(\nu^2+|\sigma_i|^2)^{\delta}$ and $\delta^2\cdot(\nu^2+|\sigma_i|^2)^{\delta}$.
  \end{proof}
\end{Proposition}

The above equation also implies that $\omega_{\delta\nu}$ converges in $C_{loc}^{\infty}$-topology to $\omega_{\delta}$ as $\nu\to 0$. The K\"ahler metrics $\omega_{\delta}$ behave well around any point of $D$ in the following sense, which is clear by construction. 
\begin{Lemma}\label{isometry}
  Let $x\in D_{J}$ with $J\subset I$, we may choose a small coordinate neighborhood centered at $x$ such that the defining function of $D$ is $z_{1}\cdots z_{k}=0$. Then there exists a positive constants $C_{1}$ such that
  \begin{equation*}
    C_{1}^{-1}\cdot\omega_{\delta}\leq \sqrt{-1}\cdot \delta^2\cdot \sum_{i=1}^{k}\frac{\diff z_{i}\wedge \diff \overline{z_{i}}}{\left|z_{i}\right|^{2-2\delta}}+\sqrt{-1}\cdot \partial\overline{\partial}|z|^{2}\leq C_{1}\cdot\omega_{\delta}.
  \end{equation*}  
\end{Lemma}
 Notice that the pullback of $\omega_{\delta}$ by $\pi$ is a degenerated conical K\"ahler metric on $\widetilde{X}\setminus D$. Fix a normalized K\"ahler metric $\omega_{\widetilde{X}}$ on $\widetilde{X}$ such that $\int_{\widetilde{X}}\omega_{\widetilde{X}}^{n}=1$. We put $\omega_{\epsilon\delta}\coloneqq \omega_{\delta}+\epsilon\cdot\omega_{\widetilde{X}}$. Then $\omega_{\epsilon\delta}$ is a conical K\"ahler metric of $\widetilde{X}\setminus D$. And we have the following proposition.
\begin{Proposition}\label{Simpsonassumption}
  The K\"ahler manifold $(\widetilde{X}\setminus D, \omega_{\epsilon\delta})$ satisfies the following three assumptions:\begin{enumerate}
    \item The volume of $\widetilde{X}\setminus D$ is uniformly bounded independent of $\epsilon$ and $\delta$.
    \item There exists an exhaustion function $\phi$ with $\sqrt{-1}\Lambda_{\omega_{\epsilon\delta}}\partial\overline{\partial}\phi$ bounded.
    \item If $f$ is a bounded positive function on $\widetilde{X}\setminus D$ such that $\Delta_{\epsilon\delta}f\leq B$ for some positive function $B\in L^{p}$ $(p>n)$, then $\|f\|_{L^{\infty}}\leq C_{2}(\|B\|_{L^{p}}+\|f\|_{L^{1}})$. The constant $C_{2}$ is independent of $\epsilon$ and $\delta$. 
  \end{enumerate}
  Here $\Delta_{\epsilon\delta}=2\sqrt{-1}\Lambda_{\omega_{\epsilon\delta}}\overline{\partial}\partial$ is the negative Laplace operator and the Laplace operators appear in the rest of the paper will always be the negative one. 
  \begin{proof}
    The first assertion follows directly from the previous lemma. We put $\phi\coloneqq \log|\sigma|$. Then the second assertion follows from the Poincaré-Lelong formula.

    To proof the third assertion, we need a very important result of \cite{Go-Ph-So-Ja} which will also be used later. Let $(X,\omega_{X})$ be a compact K\"ahler manifold with a normalized K\"ahler metric $\omega_{X}$ such that $\int_{X}\omega_{X}^{n}=1$. Suppose that the complex dimension of $X$ is $n$. Given any K\"ahler metric $\omega$ on $X$, we denote its volume by $V_{\omega}\coloneqq[\omega]^{n}$ and define the relative volume density by \begin{equation*}
      e^{\lambda_{\omega}}\coloneqq \frac{1}{V_{\omega}}\frac{\omega^{n}}{\omega_{X}^{n}}.
    \end{equation*}
    Given $p\geq 1$ we define the $p$-th Nash-Yau entropy by 
    \begin{equation*}
      \mathcal{N}_{p}(\omega)\coloneqq \frac{1}{V_{\omega}}\int_{X}\left|\log\frac{1}{V_{\omega}}\cdot\frac{\omega^n}{\omega_X^n}\right|^{p}\cdot\omega^{n}.
    \end{equation*}
    For a given nonnegative continuous function $\gamma\in C^{0}(X)$, and given parameters $0<A\leq \infty$, $K>0$, we consider the following subset of the space of K\"ahler metrics on $X$:
    \begin{equation*}
      \mathcal{W}\coloneqq \mathcal{W}(n,p,A,K,\gamma)\coloneqq \left\{\omega:[\omega]\cdot[\omega_{X}]^{n-1}<A, \mathcal{N}_{p}\leq K, e^{\lambda_{\omega}} \geq \gamma\right\}.
    \end{equation*}
    The following theorem concerning the Sobolev inequality with respect to the K\"ahler metrics in $\mathcal{W}$ was obtained in \cite[Theorem 2.1]{Go-Ph-So-Ja}.
    \begin{Theorem}[Uniform Sobolev Inequality]
      Given $p>n$ and the zero locus of $\gamma$ has Hausdorff dimension less than $2n-1$, there exist $q=q(n,p)>1$ and $C_{3}=C_{3}(n,p,A,K,\gamma,q)>0$ such that for any $\omega\in \mathcal{W}$ and any $u\in L_{1}^{2}(X)$, we have the following Sobolev-type inequality 
      \begin{equation*}
        \left(\frac{1}{V_{\omega}}\int_{X}|u-\overline{u}|^{2q}\omega^{n}\right)^{\frac{1}{q}}\leq C_{3}\frac{1}{V_{\omega}}\int_{X}|\nabla u|^{2}\omega^{n},
      \end{equation*}
      where $\overline{u}\coloneqq \frac{1}{V_{\omega}}\int_{X}u\omega^{n}$.
    \end{Theorem}
  Now we return to the proof of the third assertion. We will apply the above theorem to investigate a family of metrics on $\widetilde{X}$. We put 
  \begin{equation*}
    \omega_{\epsilon\delta\nu}\coloneqq \omega_{\delta\nu}+\epsilon\cdot\omega_{\widetilde{X}}.
  \end{equation*}
  This is a smooth family of K\"ahler metrics on $\widetilde{X}$. Since $\omega_{0\delta\nu}$ only degenerates on the exceptional divisor $\mathscr{E}$, $\omega_{\epsilon\delta\nu}$ is supported by a continuous function $\gamma$ with vanishing locus of Hausdorff dimension $2n-2$. In order to apply the above theorem, we only need to check the Nash-Yau entropy condition. Recall that $\mathscr{E}\cup D$ is simple normal crossing. Hence we may choose a local coordinate patch \(z=(z',z'',z''')\), such that \(D=\bigcup_{i=1}^{k'}\{z'_i=0\}\) and \(\mathscr{E}=\bigcup_{j=1}^{k''}\{z''_j=0\}\). We may assume that \(|z|<1\). We put 
  \begin{equation*}
    f_{\epsilon\delta\nu}\coloneqq \frac{\omega_{\epsilon\delta\nu}^n}{\omega_{\widetilde{X}}^n}.
  \end{equation*}
  Then $f_{\epsilon\delta\nu}$ either degenerates to $0$ on $\mathscr{E}$ with a speed $\prod_{j}|z''_j|^{a}$ or blows up to infinity on $D$ with a speed $\prod_i|z'_i|^{-a}$ with some uniform \(a>0\). Therefore, if we fix a \(p>n\), we have
  \[
  |\log(f_{\epsilon\delta\nu})|^p\leq \frac{C_{15}}{\delta^{k'}\cdot\prod_{i=1}^{k'}|z_i'|^{\frac{\delta}{2}}\cdot\prod_{j=1}^{k''}|z_j''|^{\frac{1}{2}}}
  \]
  with $C_{15}$ a constant independent of $\delta$, \(\epsilon\) and \(\nu\). Then it follows from Lemma~\ref{isometry} that 
  \[
  |\log(f_{\epsilon\delta\nu})|^p\cdot \omega_{\epsilon\delta\nu}^n\leq \frac{C_{15}\cdot\delta^{k'}}{\prod_{i=1}^{k'}|z_i'|^{2-\frac{3\delta}{2}}\cdot\prod_{j=1}^{k''}|z_j''|^{\frac{1}{2}}}.
  \]
  Hence the integral of $|\log(f_{\epsilon\delta\nu})|^p\cdot \omega_{\epsilon\delta\nu}^n$ is uniformly bounded. Consequently, we can find suitable constants $p>n$, $A$ and $K$ such that the family of metrics $\omega_{\epsilon\delta\nu}$ belongs to $\mathcal{W}(n,p,A,K,\gamma)$. Then the above theorem implies the following lemma.
  \begin{Lemma}\label{uniSob}
    For any $f\in C_{0}^{\infty}(\widetilde{X}\setminus D^{*})$ and with respect to any conical K\"ahler metric $\omega_{\epsilon\delta}$ on $\widetilde{X}\setminus D^{*}$, we have the uniform Sobolev inequality \begin{equation*}
      \|f\|_{L^{q}}\leq C_{4}\|f\|_{L_{1}^{2}}
    \end{equation*}
    with $C_{4}$ independent of $\epsilon$ and $\delta$.
    \begin{proof}
      This is true because $f$ is compactly supported in $\widetilde{X}\setminus D^{*}$ and $\omega_{\epsilon\delta\nu}$ converges in $C_{loc}^{\infty}$-topology to $\omega_{\epsilon\delta}$ as $\nu$ tends to $0$.
    \end{proof}
  \end{Lemma}
  Now we are ready to prove the third assertion by Moser's iteration. Let $f$ be a positive bounded function which satisfies the assumption of the third assertion, i.e. $\Delta_{\omega_{\epsilon\delta}}f=\sqrt{-1}\Lambda_{\omega_{\epsilon\delta}}\partial\overline{\partial}f\leq B$. We put $f_{\nu}\coloneqq (f+\nu\log|\sigma|)^{+}$ with $0<\nu<1$. Then we have \begin{equation*}
    \Delta_{\omega_{\epsilon\delta}}f_{\nu}\leq B_{\nu}
  \end{equation*}
  in the weak sense, where $B_{\nu}$ converges to $B$ in $C^{\infty}$-topology as $\nu$ tends to $0$. As $f_{\nu}$ lies in the Sobolev closure of $C_{0}^{\infty}(\widetilde{X}\setminus D^{*})$, hence we may apply Moser's iteration to obtain the estimate \begin{equation*}
    \|f_{\nu}\|_{L^{\infty}}\leq C_{5}(\|f_{\nu}\|_{L^{1}}+\|B_{\nu}\|_{L^{p}})
  \end{equation*}
  with $p>n$. Taking the limit as $\nu$ tends to $0$ completes the proof.
  \end{proof}
\end{Proposition}
The uniform Sobolev inequality also implies the following uniform upper bound for the heat kernels on $\widetilde{X}\setminus D$ with repsect to the family of conical metrics $\omega_{\epsilon\delta}$. 
\begin{Proposition}\label{HeatkernelUB}
  Let $K_{\epsilon\delta}$ be the heat kernel with respect to $\omega_{\epsilon\delta}$, then for any $\tau>0$, there exists a constant $C_{K}(\tau)$ which is independent of $\epsilon$ and $\delta$ such that 
  \begin{equation*}
    0<K_{\epsilon\delta}(x,y,t)\leq C_{K}(\tau)\left(t^{n}\exp(-\frac{d_{\epsilon\delta}(x,y)}{(4+\tau)t})+1\right)
  \end{equation*}
  where $d_{\epsilon\delta}(\cdot,\cdot)$ is the distance function of $\widetilde{X}\setminus D^{*}$ with repsect to the metric $\omega_{\epsilon\delta}$.
  \begin{proof}
    Recall that $\omega_{\epsilon\delta\nu}$ is a family of smooth metrics on the compact manifold $\widetilde{X}$ such that the uniform Sobolev inequality is satisfied. On the one hand, by combining the diagonal heat kernel estimate shown in~\cite{Ch-Li} and the Gaussian upper bound estimate in~\cite[Theorem 1.1]{GrYan}, we have 
    \begin{equation*}
      0<K_{\epsilon\delta\nu}(x,y,t)\leq C_{K}(\tau)\left(t^{n}\exp(-\frac{d_{\epsilon\delta\nu}(x,y)}{(4+\tau)t})+1\right).
    \end{equation*}
    On the other hand, it is easy to show that $K_{\epsilon\delta\nu}$ converges uniformly to the heat kernel $K_{\epsilon\delta}$ on compact subsets of $(\widetilde{X}\setminus D^{*})\times (\widetilde{X}\setminus D^{*})$ as $\nu$ tends to $0$. Hence the proposition follows.
  \end{proof}
\end{Proposition}
\subsection{Parabolic metric}
The target of this subsection is to construct a Hermitian metric $\widehat{H}$ on the regular part of $\pF$ which is compatible with the parabolic structure in the sense of Definiton~\ref{comatibleandadmissible}. We also wish it to satisfy the following equations in the sense of currents on $X$: 
\begin{align*}
  &\ch_{1}(\pF)=\frac{\sqrt{-1}}{2\pi}\tr(F_{\widehat{H}})\\
  &\ch_{1}(\pF)^2=\left(\frac{\sqrt{-1}}{2\pi}\right)^{2}\tr(F_{\widehat{H}})\wedge \tr(F_{\widehat{H}})\\
  &\ch_{2}(\pF)=\frac{1}{2}\left(\frac{\sqrt{-1}}{2\pi}\right)^{2}\tr(F_{\widehat{H}}\wedge F_{\widehat{H}}).
\end{align*}
\begin{Definition}
  If a Hermitian metric $H$ defined on ${\pF}_{|X^{\circ}\setminus D}$ satisfies the first equation above, we call it adapted to $\pF$ in codimension $1$. If it satisfies all of the equations, we call it adapted to $\pF$ in codimension $2$.
\end{Definition}
In view of Lemma~\ref{cherncharacterpreserve}, it suffices to construct a Hermitian metric $H_0$ on the regular part of the locally abelian parabolic bundle $E_*'\coloneqq E\otimes[-P_0]$ which is adapted in codimension $2$. Indeed, should this be completed, we may choose a Hermitian metric $h_{P_0}$ for the line bundle $\mathcal{O}_{\widetilde{X}}(P_0)$ and then $\widehat{H}\coloneqq H_0\otimes h_{P_0}$ will be a Hermitian metric on $E_{|\widetilde{X}\setminus D^*}$ which in turn will induce a metric on the regular part of $\pF$. Since $E\otimes[-P]\cong \F$, it is easy to see that $\widehat{H}$ is adapted to $\pF$ in codimension $2$ because the singular locus has codimension at least $3$. 

We introduce an appraoch to construct an adapted metric on a locally abelian parabolic bundle. The notation used in the construction process will be independent of that used elsewhere. Firstly, it is important to have a smooth decomposition of the bundle in a small neighborhood of the divisor $D$ which is in some sense compatible with the paraboic structrue. 

\subsubsection{Smooth decomposition}
\paragraph{Local patching}
Let $X$ be a complex manifold. Let $D=\bigcup_{ 1\leq i\leq \ell}D_{i}$ be a simple normal crossing divisor. We set $Y=\bigcap_{i=1}^{\ell}D_i$.
We set $S=\{0\leq a <1\}$. For any $1\leq i\leq\ell$, let $q_{i}:S^{\ell}\to S$ denote the projection onto the $i$-th component.

Let $F_*=(F,\boldsymbol{F})$ be a locally abelian parabolic bundle on $X$. Recall that $\boldsymbol{F}=({}^{i}\!F\,|\,i=1,\ldots,\ell)$ is a tuple of decreasing filtrations ${}^{i}\!F$ of $F_{|D_i}$ by holomorphic subbundles indexed by $S$ such that the following holds.

\begin{itemize}
\item For any $1\leq i\leq\ell$ and $a\in S$, there exists $\epsilon>0$ such that ${}^{i}\!F_{a}={}^{i}\!F_{a-\epsilon}$.
\item For any $P\in Y$, there exist a neighbourhood ${\U}_{P}$ of $P$ in $X$ and a holomorphic decomposition $F_{|\mathcal{U}_{P}}=\boplus_{\va\in S^{\ell}}{}^{P}\!G_{\va}$ such that \[\bigoplus_{q_{i}(\va)\geq @}{}^{P}\!G_{\va|D_i\cap{\U}_{P}}={}^{i}\!F_{|D_i\cap{\U}_{P}}.\]
\end{itemize}
Here the symbol ``@'' denotes a real variable taking values from $[0,1[$ and we regard $$\bigoplus_{q_{i}(\va)\geq @}{}^{P}\!G_{\va|D_i\cap{\U}_{P}}$$ as a left continuous filtration.

Let $X_{2}\subset X_{1}\subset X$ be open subsets such that the closure of $X_{2}$ is contained in $X_{1}$. Suppose that there exists a $C^{\infty}$-decomposition
\[F_{|X_{1}}=\bigoplus_{\boldsymbol{a}\in S^{\ell}}{}^{1}\!G_{\boldsymbol{a}}\]
such that the following holds:
\[\bigoplus_{q_{i}(\boldsymbol{a})\geq @}{}^{1}\!G_{\va|D_i\cap X_{1}}={}^{i}\!F_{|D_i\cap X_{1}}.\]

\begin{Proposition}\label{localpatching}
There exist a neighbourhood ${\U}$ of $Y$ and a $C^{\infty}$-decomposition
\[F_{|\mathcal{U}}=\bigoplus_{\boldsymbol{a}\in S^{\ell}}G_{\boldsymbol{a}}\]
such that the following holds:
\begin{itemize}
\item $\bigoplus_{q_{i}(\boldsymbol{a})\geq @}G_{\boldsymbol{a}|D_i\cap{\U}}={}^{i}\!F_{|D_i\cap{\U}}$.
\item $G_{\boldsymbol{a}|X_{2}}={}^{1}\!G_{\boldsymbol{a}|X_{2}}$.
\end{itemize}
\end{Proposition}

\begin{proof}
Since locally, $F_*$ is a direct sum of parabolic line bundles, there exists a tuple of open subsets ${\U}_{k}$ ($k\in\Gamma$) of $X$ such that the following holds:
\begin{itemize}
\item $Y\subset\mathcal{V}:=\bigcup{\U}_{k}\cup X_{1}$
\item On each $\mathcal{U}_{k}$, there exists a $C^{\infty}$-decomposition
\[F_{|\mathcal{U}_{k}}=\bigoplus_{\boldsymbol{a}\in S^{\ell}}{}^{k}\!G_{\boldsymbol{a}}\]
such that the following holds:
\[\bigoplus_{q_{i}(\boldsymbol{a})\geq @}{}^{k}\!G_{\boldsymbol{a}|D_i\cap\mathcal{U}_{k}}={}^{i}\!F_{|D_i\cap\mathcal{U}_{k}}.\]
\end{itemize}

Let $\{\chi_{k}\}\cup\{\chi_{X_{1}}\}$ be a partition of unity on $\mathcal{V}$ subordinate to the open covering $\mathcal{V}=\bigcup\mathcal{U}_{k}\cup X_{1}$. We may assume that $\chi_{X_{1}}=1$ on $X_{2}$. Let $\psi:S^{\ell}\to\mathbb{C}$ be an injection. We set \[f_{k}=\sum_{\boldsymbol{a}\in S^{\ell}}\psi(\boldsymbol{a})\cdot\id_{{}^{k}\!G_{\boldsymbol{a}}}\] and \[f_{X_{1}}=\sum_{\boldsymbol{a}\in S^{\ell}}\psi(\boldsymbol{a})\cdot\id_{{}^{1}\!G_{\boldsymbol{a}}}.\] We obtain the following $C^{\infty}$-endomorphism of $F$ on $\mathcal{V}$:
\[f=\sum_{k\in\Gamma}\chi_{k}\cdot f_{k}+\chi_{X_{1}}\cdot f_{X_{1}}.\]
In the mean time, we obtain the following lemma by the construction.

\begin{Lemma}\label{endpreserving}
The restriction $f_{|Y}$ preserves the filtrations ${}^{i}\!F$ ($i=1,\ldots,\ell$) of $F_{|Y}$. Moreover, the induced endomorphism on ${}^{\underline{\ell}}\!\Gr_{\boldsymbol{a}}(F_{|Y})$ equals the multiplication of $\psi(\boldsymbol{a})$.
\end{Lemma}

As a result, we obtain the eigen decomposition
\[(F,f)_{|Y}=\bigoplus_{\alpha\in\mathbb{C}}(K_{\alpha},\alpha\cdot\id_{K_{\alpha}}).\]
There exists a neighbourhood $\mathcal{U}$ of $Y$ and a decomposition
\[(F,f)_{|\mathcal{U}}=\bigoplus_{\alpha\in\mathbb{C}}(K_{\mathcal{U},\alpha},f_{\alpha})\]
such that $K_{\mathcal{U},\alpha|Y}=K_{\alpha}$. If $\mathcal{U}$ is sufficiently small, any eigenvalues of $f_{\alpha|P}$ ($P\in\mathcal{U}$) are close to $\alpha$. In particular, we may assume that there are no common eigenvalues of $f_{\alpha}$ and $f_{\beta}$ ($\alpha\neq\beta$). We set
\[G_{\boldsymbol{a}}=K_{\mathcal{U},\psi(\boldsymbol{a})}.\]
The restriction $f_{|D_i\cap\mathcal{U}}$ preserves the filtration ${}^{i}\!F$. We have
\[{}^{i}\!F_{b|Y}=\bigoplus_{q_{i}(\boldsymbol{a})\geq b}K_{\psi(\boldsymbol{a})}.\]
Then, we obtain $\bigoplus_{q_{i}(\boldsymbol{a})\geq @}G_{\boldsymbol{a}|D_i\cap\mathcal{U}}={}^{i}\!F_{|D_i\cap\mathcal{U}}$. Because $\chi_{X_{1}}=1$ on $X_{2}$, we have $G_{\boldsymbol{a}|X_{2}}={}^{1}\!G_{\boldsymbol{a}|X_{2}}$.
\end{proof}

\paragraph{Glocal decomposition}
With the above proposition at hand it is easy to get a global $C^{\infty}$-decomposition near $D$.
Let $X$ be a complex manifold. Let $D=\bigcup_{i\in \Lambda}D_{i}$ be a simple normal crossing divisor. For any $I\subset\Lambda$, we set $D_{I}=\bigcap_{i\in I}D_i$, $\partial D_{I}=\bigcup_{j\in\Lambda\setminus I}(D_{I}\cap D_{j})$ and $D_{I}^{\circ}=D_{I}\setminus\partial D_{I}$. For any $I\subset J\subset\Lambda$, let $q_{I,J}:S^{J}\to S^{I}$ denote the projection. As before, let $F_{*}=(F,\boldsymbol{F})$ be a parabolic bundle on $X$. Let $\boldsymbol{F}=({}^{i}\!F\,|\,i\in\Lambda)$ be the parabolic structure.

\begin{Proposition}\label{globaldecomposition}
There exist neighbourhoods ${\U}_{i}$ ($i\in\Lambda$) of $D_i$ and $C^{\infty}$-decompositions
\[F_{|{\U}_{i}}=\bigoplus_{a\in S}{}^{i}\!G_{a}\]
such that the following holds:
\begin{itemize}
\item $\bigoplus_{a\geq b}{}^{i}\!G_{a|D_i}={}^{i}\!F_{b}$ holds for any $b\in S$.
\item For any $I\subset\Lambda$ and $\boldsymbol{a}=({}^{i}\!a\,|\,i\in I)\in S^{I}$, on ${\U}_{I}=\bigcap_{i\in I}{\U}_{i}$, we set ${}^{I}\!G_{\boldsymbol{a}}=\bigcap_{i\in I}{}^{i}\!G_{{}^{i}\!a|{\U}_{I}}$. Then, $F_{|{\U}_{I}}=\bigoplus_{\boldsymbol{a}\in S^{I}}{}^{I}\!G_{\boldsymbol{a}}$ holds.
\end{itemize}

\begin{proof}
By using a descending induction on $|I|$, we shall construct neighbourhoods ${\cal V}_{I}$ ($I\subset\Lambda$) of $D_{I}$ and decompositions
\[F_{|{\cal V}_{I}}=\bigoplus_{\boldsymbol{a}\in S^{I}}{}^{I}\!G_{\boldsymbol{a}}\]
such that the following holds:
\begin{itemize}
\item[(a)] $\bigoplus_{q_{i,I}(\boldsymbol{a})\geq b}{}^{I}\!G_{\boldsymbol{a}|D_i\cap{\cal V}_{I}}={}^{i}\!F_{b|D_i\cap{\cal V}_{I}}$ holds for any $b\in S$.
\item[(b)] For $I\subset J$ and $\boldsymbol{a}\in S^{I}$, we have ${}^{I}\!G_{\boldsymbol{a}|{\cal V}_{J}}=\bigoplus_{q_{I,J}(\boldsymbol{b})=\boldsymbol{a}}{}^{J}\!G_{\boldsymbol{b}}$.
\end{itemize}

Suppose that we have already constructed such decompositions for $J\subset\Lambda$ with $|J|\geq k_{0}+1$. Let $I\subset\Lambda$ with $|I|=k_{0}$.
For $I\subset J\subset\Lambda$ with ${\U}_{J}\neq\varnothing$ and for $\boldsymbol{a}\in S^{I}$, we set
\begin{equation*}
{}^{J}\!G_{\boldsymbol{a}}\coloneqq\bigoplus_{q_{I,J}(\boldsymbol{b})=\va}{}^{J}\!G_{\boldsymbol{b}}.
\end{equation*}
We obtain the following decomposition:
\[F_{|{\U}_{J}}=\bigoplus_{\boldsymbol{a}\in S^{I}}{}^{J}\!G_{\boldsymbol{a}}\]
By the condition (b), we obtain ${}^{J}\!G_{\boldsymbol{a}|{\U}_{K}}={}^{K}\!G_{\boldsymbol{a}}$ for any $J\subset K$ and $\boldsymbol{a}\in S^{I}$. From ${}^{J}\!G_{\boldsymbol{a}}$ ($I\subsetneq J$), we obtain a $C^{\infty}$-subbundle $^{1,I}\!G_{\boldsymbol{a}}$ of $F$ on ${\U}_{I,1}=\bigcup_{I\subsetneq J}{\U}_{J}$. We obtain the decomposition
\[F_{|{\U}_{I,1}}=\bigoplus_{\boldsymbol{a}\in S^{I}}{}^{1,I}\!G_{\boldsymbol{a}}\]
Let ${\U}_{I,2}\subset{\U}_{I,1}$ be an open neighbourhood of $\partial D_{I}$ whose closure is contained in ${\U}_{I,1}$. By Proposition~\ref{localpatching}, there exist a neighbourhood ${\cal V}_{I}$ of $D_{I}$ and a decomposition
\[F_{|{\U}_{I}}=\bigoplus_{\boldsymbol{a}\in S^{I}}{}^{I}\!G_{\boldsymbol{a}}\]
such that the condition (a) holds for $I$ and that ${}^{I}\!G_{\boldsymbol{a}|{\U}_{I,2}}={}^{1,I}\!G_{\boldsymbol{a}|{\U}_{I,2}}$. We replace ${\cal V}_{J}$ with ${\cal V}_{J}\cap{\U}_{I}$ for any $J\ni I$. Then, we obtain the claim of the proposition.
\end{proof}
\end{Proposition}

\subsubsection{Construction of metric}
Let $H_{1}$ be a Hermitian metric of $F$ such that the following holds:
\begin{itemize}
\item The decomposition $F_{|{\U}_{I}}=\bigoplus {}^{I}\!G_{\boldsymbol{a}}$ is orthogonal with respect to $H_{1|{\U}_{I}}$.
\end{itemize}
Let $h^{(i)}$ ($i\in\Lambda$) be Hermitian metrics of ${\cal O}_{X}(D_i)$. We set $\tau_{i}=h^{(i)}(1,1)$. We may assume that $\tau_{i}$ is constantly $1$ on $X\setminus{\U}_{i}$. Let $\tau\coloneqq \bigotimes_{i\in \Lambda}\tau_i$ be the induced square length of the canonical section of $D$. Let $s_{i}$ be the automorphism of $F_{|X\setminus D_i}$ such that the following holds:
\begin{itemize}
\item $s_{i}=\id$ on $X\setminus{\U}_{i}$.
\item $s_{i}=\bigoplus_{{a}\in S}\tau_{i}^{-a}\id_{{}^{i}\!G_{a}}$ on ${\U}_{i}$.
\end{itemize}

We obtain the automorphism $s=\prod_{i\in\Lambda}s_{i}$ of $F$ which is self-adjoint with respect to $H_{1}$. We define the $C^{\infty}$-Hermitian metric $H_{0}$ of $F_{|X\setminus D}$ by $H_{0}(u,v)=H_{1}(su,v)$ for any local sections $u,v$ of $F$. And we may extend $H_{0}$ smoothly to $X$.


\subsubsection{Adaptedness}
It follows from the construction that for any $[\eta]\in H^{n-1,n-1}(X,\mathbb{C})$ we have:
\begin{Proposition}\label{adaptcod1bundle}
    $\ch_{1}(F_{*})\cdot [\eta]=\frac{\sqrt{-1}}{2\pi}\int_{X\setminus D}\tr(F_{\nabla_{H_0}})\wedge \eta$.
\end{Proposition}

Let $\varepsilon$ be the largest gap of the weights and $\kappa$ be the smallest gap of the weights. 

Let $(U,z_{1},\ldots,z_{n})$ be a holomorphic coordinate neighbourhood of $X$ such that $D\cap U=\bigcup_{i=1}^{\ell}\{z_{i}=0\}$.
For $1\leq i\leq\ell$, we have $\lambda(i)\in\Lambda$ such that $D_{U,\lambda(i)}\coloneqq D_{\lambda(i)}\cap U=\{z_{i}=0\}$. We set $I=\{\lambda(i)\,|\,i=1,\ldots,\ell\}$. By shrinking $U$, we assume $U\subset\U_{I}$. We obtain the induced $C^{\infty}$-decomposition
\[F_{|U}=\bigoplus_{\boldsymbol{a}\in S^{I}}{}^{I}\!G_{\boldsymbol{a}}.\]
For any $1\leq k\leq\ell$ and $a\in S$, we set
\[{}^{k}\!G_{a}=\bigoplus_{q_{\lambda(k)}(\boldsymbol{a})=a}{}^{I}\!G_{\boldsymbol{a}}.\]
We obtain the decomposition
\begin{equation}\label{decomeq}
  F_{|U}=\bigoplus_{a\in S}{}^{k}\!G_{a}.
\end{equation}

Let \(\iota_{\va}\) be the injection from ${}^{I}\!G_{\va}$ to $F$ and \(\pi_{\va}\) be the orthogonal projection from $F$ to ${}^{I}\!G_{\va}$ with respect to the metric $H_{1}$. We set $\nabla_{\va}\coloneqq \pi_{\va}\cdot \nabla_{H_1}\cdot \iota_{\va}$ as the induced connection on ${}^{I}\!G_{\va}$. Then $\widetilde{\nabla}\coloneqq \boplus_{\va}\nabla_{\va}$ is a unitary connection which preserves the decomposition. We have for any $a$ and $k$, restricting on ${}^{k}\!G_{a|D_k}={}^{k}\Gr_a F_{*}$, $\widetilde{\nabla}=\nabla_{H_1}$.

We consider the decomposition
    \[
    \nabla_{H_1}=\widetilde{\nabla}+\Psi
    \] 
where $\Psi$ is a section of
\[\bigoplus_{\va\neq \vb}A^{1}(\operatorname{Hom}({}^{I}\!G_{\vb},{}^{I}\!G_{\va})).\]

Let $\Psi=\Psi^{1,0}+\Psi^{0,1}$ be the decomposition into the $(1,0)$-part and the $(0,1)$-part. We express
\[\Psi^{1,0}=\sum \Psi_{j}^{1,0}\,dz_{j},\quad \Psi^{0,1}=\sum \Psi_{j}^{0,1}d\overline{z}_{j}.\]
We have the decompositions
\[\Psi_{j}^{p,q}=\sum(\Psi_{j}^{p,q})_{\va,\vb},\qquad (\Psi_{j}^{p,q})_{\va,\vb}\in A^{0}(\operatorname{Hom}({}^{I}\!G_{\vb},{}^{I}\!G_{\va})).\]
By the construction, $(\Psi_{j}^{p,q})_{\va,\va}=0$ for any $1\leq j\leq n$ and $\va\in S^I$.

We denote $F_{\nabla}$ as the curvature tensor with respect to some connection $\nabla$. Then we have the following lemma.
\begin{Lemma}\label{curvatureboundbundle}
    $|F_{\nabla_{H_0}}|_{H_0}=O(\tau^{-\frac{1+\varepsilon}{2}})$
\end{Lemma}

\begin{proof}
    It suffices to consider in $U$. 
    We have 
    \[F_{\nabla_{H_1}}=F_{\widetilde{\nabla}}+\widetilde{\nabla}(\Psi)+\Psi\wedge \Psi.\]

    Hence 
    \begin{align*}
        F_{\nabla_{H_0}}=&F_{\nabla_{H_1}}+\overline{\partial}(s^{-1}\nabla_{H_1}^{1,0}s)\\
        =&F_{\widetilde{\nabla}}+\widetilde{\nabla}(\Psi)+\Psi\wedge \Psi+\overline{\partial}(s^{-1}\wnablayl s+s^{-1}[\Psi,s])\\
        =&B+\overline{\partial}(s^{-1}\Psi s)
    \end{align*}
    where $B$ is smooth. We put $A\coloneqq s^{-1}\Psi s$. Then we have $(\overline{\partial}A)_{\va,\vb}=O(|\tau_I|^{\vb-\va-\boldsymbol{\frac{1}{2}}})$ where $\boldsymbol{\frac{1}{2}}\in S^I$ . Since $H_0$ is diagonal with respect to the decomposition, it suffices to estimate \[\left((\overline{\partial}A)_{\va,\vb}\cdot (\overline{\partial}A)_{\va,\vb}\cdot H_{2,\va,\va}\cdot H_{2,\vb,\vb}^{-1}\right)^{\frac{1}{2}}.\]
    Simple calculation shows that \[
    \left((\overline{\partial}A)_{\va,\vb}\cdot (\overline{\partial}A)_{\va,\vb}\cdot H_{2,\va,\va}\cdot H_{2,\vb,\vb}^{-1}\right)^{\frac{1}{2}}=O(\tau_I^{\frac{\vb-\va-\boldsymbol{1}}{2}})=O(\tau^{-\frac{1+\varepsilon}{2}}).
    \]
\end{proof}

Next we show the adaptedness in codimension $2$. 

We still consider in $U$ but fix a $1\leq k \leq \ell$.
We consider the decomposition
\[\nabla_{H_{1}}=\widetilde{\nabla}+\Psi,\]
as above. 

Let $\Psi=\Psi^{1,0}+\Psi^{0,1}$ be the decomposition into the $(1,0)$-part and the $(0,1)$-part. We express
\[\Psi^{1,0}=\sum \Psi_{j}^{1,0}\,dz_{j},\quad \Psi^{0,1}=\sum \Psi_{j}^{0,1}d\overline{z}_{j}.\]
We have the decompositions
\[\Psi_{j}^{p,q}=\sum(\Psi_{j}^{p,q})_{a,b},\qquad (\Psi_{j}^{p,q})_{a,b}\in A^{0}(\operatorname{Hom}(G_{b},G_{a})).\]
By the construction, $(\Psi_{j}^{p,q})_{a,a}=0$ for any $1\leq j\leq n$ and $a\in S$.

\begin{Lemma}
For $j\neq k$ and $a<b$, we have
\[(\Psi_{j}^{0,1})_{a,b}=O\big(|z_{k}|\big).\]
\end{Lemma}
\begin{proof}
Note that ${}^{k}\!F$ is a left continuous decreasing filtration of $F_{|D_{k}}$ by holomorphic subbundles. Hence for any $b\in S$,
\[(\Psi_{j}^{0,1})_{|D_{k}\cap U}({}^{k}\!F_{b})\subset{}^{k}\!F_{\geq b}.\] It implies the claim of the lemma.
\end{proof}
Because $\Psi^{0,1}=-(\Psi^{1,0})_{H_{1}}^{\dagger}$, we obtain the following for $j\neq k$ and $a>b$:
\[(\Psi_{j}^{1,0})_{a,b}=O\big(|z_{k}|\big).\]
As a result, we obtain the following lemma.

\begin{Lemma}\label{Psiestimat}
 If $a>b$ and $j\neq k$, then
      \[ (s^{-1}\Psi_{j}^{1,0}s)_{a,b}=O\Bigl(|z_k|^{1-2\varepsilon}\prod_{i\neq k}|z_i|^{-2\varepsilon}\Bigr). \]
If $a<b$, then for any $j$,
     \[ (s^{-1}\Psi_{j}^{1,0}s)_{a,b}=O\Bigl(\prod_{\substack{1\leq i\leq \ell}}|z_i|^{-2\varepsilon}\Bigr). \]
\end{Lemma}

Let $F_{\nabla_{H_{1}}}$ denote the curvature of $\nabla_{H_{1}}$. Let $F_{\widetilde{\nabla}}$ denote the curvature of $\widetilde{\nabla}$. We obtain
\[F_{\nabla_{H_{1}}}=F_{\widetilde{\nabla}}+\widetilde{\nabla}(\Psi)+\Psi\wedge\Psi.\]

Let $\widetilde{\nabla}=\widetilde{\nabla}^{1,0}+\widetilde{\nabla}^{0,1}$ denote the decomposition into the $(1,0)$-part and the $(0,1)$-part. We have
\[s^{-1}\nabla_{H_{1}}^{1,0}(s)=s^{-1}\widetilde{\nabla}^{1,0}s+s^{-1}[\Psi^{1,0},s].\]

We also obtain
\begin{align}\label{p1}
\overline{\partial}(s^{-1} \nabla_{H_{1}}^{1,0} s)
&= \widetilde{\nabla}^{0,1}(s^{-1} \widetilde{\nabla}^{1,0}s)
+ \widetilde{\nabla}^{0,1}(s^{-1}[\Psi^{1,0},s]) \notag \\
&\quad + [\Psi^{0,1}, s^{-1} \widetilde{\nabla}^{1,0}s]
+ [\Psi^{0,1}, s^{-1}[\Psi^{1,0},s]].
\end{align}

Because the trace is $0$ for any section of $\operatorname{Hom}(G_{a},G_{b})$ ($a\neq b$), we obtain
\begin{align}\label{p2}
\tr\big(F_{\nabla_{H_{1}}}\cdot s^{-1} \nabla_{H_1}^{1,0}s\big)
&= \tr\Big(
  F_{\widetilde{\nabla}}\cdot s^{-1} \widetilde{\nabla}^{1,0}s
  + \widetilde{\nabla}(\Psi)\cdot s^{-1}[\Psi^{1,0}, s]
  \notag \\
&\quad
  + (\Psi \wedge \Psi)\cdot(s^{-1} \widetilde{\nabla}^{1,0}s)
  + (\Psi \wedge \Psi)\cdot s^{-1}[\Psi^{1,0}, s]
\Big).
\end{align}

We also obtain
\begin{align}\label{p3}
\tr\big( \overline{\partial}_F(s^{-1} \nabla_{H_1}^{1,0}s) \cdot s^{-1} \nabla_{H_1}^{1,0}s \big)
&= \tr\Big(
  \widetilde{\nabla}^{0,1}(s^{-1} \widetilde{\nabla}^{1,0}s) \cdot s^{-1} \widetilde{\nabla}^{1,0}s
\Big) \notag \\
&\quad + \tr\Big(
  \widetilde{\nabla}^{0,1}(s^{-1}[\Psi^{1,0},s]) \cdot s^{-1}[\Psi^{1,0},s]
\Big) \notag \\
&\quad + \tr\Big(
  [\Psi^{0,1}, s^{-1} \widetilde{\nabla}^{1,0}s] \cdot s^{-1}[\Psi^{1,0},s]
\Big) \notag \\
&\quad + \tr\Big(
  [\Psi^{0,1}, s^{-1}[\Psi^{1,0},s]] \cdot s^{-1} \widetilde{\nabla}^{1,0}s
\Big) \notag \\
&\quad + 2\tr\Big(
  \Psi^{0,1} \cdot \big(s^{-1} [\Psi^{1,0},s]\big)^2
\Big).
\end{align}


Note that
\begin{align}\label{p4}
\tr\big( 
  \widetilde{\nabla}^{0,1}(s^{-1}[\Psi^{1,0},s]) \cdot s^{-1}[\Psi^{1,0},s] 
\big)
&= \tr\Big(
  s^{-1}[\widetilde{\nabla}^{0,1}\Psi^{1,0}, s] \cdot s^{-1}[\Psi^{1,0},s]
\Big) \notag \\
&\quad - \tr\Big(
  [s^{-1} \Psi^{1,0} s,\ s^{-1} \widetilde{\nabla}^{0,1}s] \cdot s^{-1}[\Psi^{1,0},s]
\Big).
\end{align}
\begin{align}\label{p5}
\tr\big(
  [\Psi^{0,1}, s^{-1} \widetilde{\nabla}^{1,0} s] \cdot s^{-1}[\Psi^{1,0}, s]
\big)
&= - \tr\big(
  s^{-1} \widetilde{\nabla}^{1,0} s \cdot [\Psi^{0,1}, \Psi^{1,0}]
\big) \notag \\
&\quad + \tr\big(
  [\Psi^{0,1}, s^{-1} \Psi^{1,0} s] \cdot s^{-1} \widetilde{\nabla}^{1,0} s
\big).
\end{align}

\begin{align}\label{p6}
\tr\big(
  [\Psi^{0,1}, s^{-1}[\Psi^{1,0}, s]] \cdot s^{-1} \widetilde{\nabla}^{1,0} s
\big)
&= - \tr\big(
  [\Psi^{0,1}, \Psi^{1,0}] \cdot s^{-1} \widetilde{\nabla}^{1,0} s
\big) \notag \\
&\quad + \tr\big(
  [\Psi^{0,1}, s^{-1} \Psi^{1,0} s] \cdot s^{-1} \widetilde{\nabla}^{1,0} s
\big).
\end{align}


For any $2$-form $\tau$, let $\tau^{1,1}$ denote the $(1,1)$-part of $\tau$. We obtain
\begin{align}\label{p7}
&2 \tr\big(F_{\nabla_{H_{1}}} s^{-1} \nabla_{H_{1}}^{1,0} s\big)
+ \tr\big(\overline{\partial}(s^{-1} \nabla_{H_{1}}^{1,0} s) \nabla_{H_{1}}^{1,0} s\big) \notag \\
&= \tr\Big(
  2 F_{\widetilde{\nabla}}^{1,1} s^{-1} \widetilde{\nabla}^{1,0} s
  + \widetilde{\nabla}^{0,1}(s^{-1} \widetilde{\nabla}^{1,0} s) s^{-1} \widetilde{\nabla}^{1,0} s
\Big) \notag \\
&\quad + \tr\Big(
  2 \widetilde{\nabla}(\Psi)^{1,1} s^{-1} [\Psi^{1,0}, s]
  + 2 (\Psi \wedge \Psi)^{1,1} s^{-1} [\Psi^{1,0}, s]
\Big) \notag \\
&\quad + \tr\Big(
  s^{-1} [\widetilde{\nabla}^{0,1} \Psi^{1,0}, s] s^{-1} [\Psi^{1,0}, s]
  - [s^{-1} \Psi^{1,0} s, s^{-1} \widetilde{\nabla}^{0,1} s] s^{-1} [\Psi^{1,0}, s]
\Big) \notag \\
&\quad + 2 \tr\Big(
  [\Psi^{0,1}, s^{-1} \Psi^{1,0} s] s^{-1} \widetilde{\nabla}^{1,0} s
\Big) + 2 \tr\Big(
  \Psi^{0,1} \big(s^{-1} [\Psi^{1,0}, s]\big)^2
\Big).
\end{align}


We set
\[Z_{U,\lambda(k)}(\delta)=U\cap\{\tau_{\lambda(k)}=\delta\}\cap\bigcap_{\begin{subarray}{c}1\leq i\leq\ell\\ i\neq k\end{subarray}}\{\tau_{\lambda(i)}\geq\delta\}.\]

 Let $[\eta]\in H^{n-2,n-2}(X,\mathbb{C})$.By using Lemma~\ref{Psiestimat} and \eqref{p7}, we obtain
\begin{align}\label{p8}
&\lim_{\delta \to 0} \int_{Z_{U, \lambda(k)}(\delta)} \Big(
  2 \tr\big(F_{\nabla_{H_{1}}} s^{-1} \nabla_{H_{1}}^{1,0} s\big)
  + \tr\big(\overline{\partial}(s^{-1} \nabla_{H_{1}}^{1,0} s) \nabla_{H_{1}}^{1,0} s\big)
\Big)\cdot \eta \notag \\
&= \lim_{\delta \to 0} \int_{Z_{U, \lambda(k)}(\delta)} \tr\Big(
  2 F_{\widetilde{\nabla}}^{1,1} s^{-1} \widetilde{\nabla}^{1,0} s
  + \widetilde{\nabla}^{0,1}(s^{-1} \widetilde{\nabla}^{1,0} s) s^{-1} \widetilde{\nabla}^{1,0} s
\Big)\cdot \eta.
\end{align}

We set \( s_{\neq k} = \prod_{i \neq k} s_i \). It follows that \( s = s_{\neq k} \cdot s_k = s_k \cdot s_{\neq k} \), and consequently:
\begin{align}
s^{-1}\widetilde{\nabla}^{1,0}s = s_{\neq k}^{-1}\widetilde{\nabla}^{1,0}s_{\neq k} + s_k^{-1}\widetilde{\nabla}^{1,0}s_k.
\end{align}

The limit expression becomes:
\begin{align}
\lim_{\delta \to 0} \int_{Z_{U,\lambda(k)}(\delta)} &\tr\Big(\widetilde{\nabla}^{0,1}(s^{-1}\widetilde{\nabla}^{1,0}s)s^{-1}\widetilde{\nabla}^{1,0}s\Big)\eta \nonumber \\
&= \lim_{\delta \to 0} \int_{Z_{U,\lambda(k)}(\delta)} \tr\bigg(\Big(\widetilde{\nabla}^{0,1}(s_k^{-1}\widetilde{\nabla}^{1,0}s_k) \nonumber \\
&\quad + \widetilde{\nabla}^{0,1}(s_{\neq k}^{-1}\widetilde{\nabla}^{1,0}s_{\neq k})\Big)s_k^{-1}\widetilde{\nabla}^{1,0}s_k\bigg)\cdot\eta. 
\end{align}

Define \( {}^{k}\!\Gamma = \bigoplus (-a) \cdot \id_{{}^{k}\!G_a} \). The following identities hold:
\begin{align}
s_k^{-1}\widetilde{\nabla}^{1,0}(s_k) &= {}^{k}\!\Gamma \cdot \partial \log \tau_k, \\
\widetilde{\nabla}^{0,1}\big(s_k^{-1}\widetilde{\nabla}^{1,0}(s_k)\big) &= {}^{k}\!\Gamma (\overline{\partial}\partial \log \tau_k).
\end{align}
For \( {}^k\!\mathrm{Gr}_a F_* = {}^k\!F_a / {}^k\!F_{>a} \), we have \( \tr {}^{k}\!\Gamma^2 = \sum_{a \in S} a^2 \rank({}^k\!\mathrm{Gr}_a F_*) \). The limit evaluates to:
\begin{align}
\lim_{\delta \to 0} \int_{Z_{U,\lambda(k)}(\delta)} &\tr\Big(\widetilde{\nabla}^{0,1}(s_k^{-1}\widetilde{\nabla}^{1,0}s_k) \cdot s_k^{-1}\widetilde{\nabla}^{1,0}s_k\Big)\cdot\eta \nonumber \\
&= \pm \sum_a 2\pi\sqrt{-1} \int_{U \cap H_{\lambda(k)}} a^2 \rank({}^k\!\mathrm{Gr}_a^F(E)) \overline{\partial}\partial(\log \tau_k) \cdot\eta.
\end{align}

The \( C^\infty \)-decomposition \( F_{|D_k} = \boplus {}^{k}\!G_{a} \) induces an isomorphism \( F_{|D_k} \simeq \boplus {}^k\!\mathrm{Gr}_a F_* \). The restriction of \( \widetilde{\nabla} \) to \( F_{|D_k \cap U} \) coincides with the Chern connection of \( \boplus {}^k\!\mathrm{Gr}_a F_* \) under this isomorphism.

Let \( {}^{k}\!H_{1,a} \) be the Hermitian metric on \( {}^k\!\mathrm{Gr}_a F_* \) induced by \( H_1 \). Then:
\begin{align}
\lim_{\delta \to 0} \int_{Z_{U,\lambda(k)}(\delta)} &\tr\Big(F_{\widetilde{\nabla}}^{1,1}s^{-1}\widetilde{\nabla}^{1,0}s\Big)\cdot\eta \nonumber \\
&= \pm \sum_a 2\pi\sqrt{-1} \int_{U \cap H_{\lambda(k)}} (-a) \tr F_{\nabla_{{}^{k}\!H_{1,a}}} \cdot\eta. 
\end{align}

Let \( {}^{k}\!\widetilde{H}_{1,a} \) be the Hermitian metric on \( {{}^k\!\mathrm{Gr}_a F_*}_{|_{D^\circ_k}} \) induced by \( H_1 \cdot s_{\neq k} \). We derive:
\begin{align}
\lim_{\delta \to 0} \int_{Z_{U,\lambda(k)}(\delta)} &\tr\bigg(\Big(F_{\widetilde{\nabla}}^{1,1} + \widetilde{\nabla}^{0,1}(s_{\neq k}^{-1}\widetilde{\nabla}^{1,0}s_{\neq k})\Big)s^{-1}\widetilde{\nabla}^{1,0}s\bigg)\cdot\eta \nonumber \\
&= \pm \sum_a 2\pi\sqrt{-1} \int_{U \cap H_{\lambda(k)}} (-a) \tr F_{{}^{k}\!\widetilde{H}_{1,a}}\cdot \eta. 
\end{align}

Combining these results yields:
\begin{align}\label{p11}
\lim_{\delta \to 0} \pm \int_{Z_{U,\lambda(k)}(\delta)} &\Big(2 \tr\big(F_{\nabla_{H_1}}s^{-1}\nabla_{H_1}^{1,0}s\big) + \tr\big(\overline{\partial}(s^{-1}\nabla_{H_1}^{1,0}s)s^{-1}\nabla_{H_1}s\big)\Big)\cdot\eta \nonumber \\
&= \sum_a 2\pi\sqrt{-1} \int_{U \cap H_{\lambda(k)}} (-a) \Big(\tr F_{\nabla_{{}^{k}\!H_{1,a}}} + \tr F_{\nabla_{{}^{k}\!\widetilde{H}_{1,a}}}\Big)\cdot \eta \nonumber \\
&\quad + \sum_a 2\pi\sqrt{-1} \int_{U \cap H_{\lambda(k)}} a^2 \rank\big({}^k\!\mathrm{Gr}_a F_*\big) \overline{\partial}\partial(\log \tau_k)\cdot \eta. 
\end{align}

Similarly, fix $1\leq j,k\leq\ell$. Then we take the limit after integrating by part on the set \[Z_{U,\lambda(j),\lambda(k)}\coloneqq U\cap \{\tau_{\lambda(j)}=\tau_{\lambda(k)}=\delta\}\cap\bigcap_{\substack{1\leq i\leq \ell\\i\neq j,k}}\{\tau_{\lambda_i\geq\delta}\}.\]
we could obtain
\begin{align}\label{p12}
\lim_{\delta \to 0} \pm \int_{Z_{U,\lambda(j),\lambda(k)}(\delta)} &\Big(2 \tr\big(F_{\nabla_{H_1}}s^{-1}\nabla_{H_1}^{1,0}s\big) + \tr\big(\overline{\partial}(s^{-1}\nabla_{H_1}^{1,0}s)s^{-1}\nabla_{H_1}s\big)\Big)\cdot\eta \nonumber \\
&= \sum_{i\in\{j,k\}}\sum_{{}^{i}\!a} 2\pi\sqrt{-1} \int_{U \cap D_{\lambda(i)}} (-{}^{i}\!a) \Big(\tr F_{\nabla_{{}^{i}\!H_{1,{}^{i}\!a}}} + \tr F_{\nabla_{{}^{i}\!\widetilde{H}_{1,{}^{i}\!a}}}\Big)\cdot \eta \nonumber \\
&\quad + \sum_{i\in\{j,k\}}\sum_{{}^{i}\!a} 2\pi\sqrt{-1} \int_{U \cap D_{\lambda(i)}} {}^{i}\!a^2 \rank\big({}^{i}\!\mathrm{Gr}_a F_*\big) \overline{\partial}\partial(\log \tau_{\lambda(i)})\cdot \eta \nonumber \\
&\quad -\sum_{\va\in S^{(j,k)}} (2\pi)^2\int_{U\cap D_{j,k}}{}^{j}\!a\cdot{}^{k}\!a\cdot\rank({}^{jk}\!\Gr_{\va}F_{*})\cdot\eta
\end{align}
Combining equations~\eqref{p11} and \eqref{p12} we get the adaptedness in codimension $2$, i.e., the following proposition.
\begin{Proposition}\label{adaptedcodim2bundle}
  \(\ch_2(F_*)\cdot \eta=\frac{1}{2}\left(\frac{\sqrt{-1}}{2\pi}\right)^2\int_{X\setminus D}\tr(F_{\nabla_{H_0}}\wedge F_{\nabla_{H_0}})\wedge \eta.\)
\end{Proposition}

Now return to our previous setting, i.e. a parabolic sheaf \(\pF\) over \((X,D)\). As we have remarked in the beginning of this subsection, we can construct a metric $\widehat{H}$ on the regular part of $\pF$ which is adapted to it in codimension \(2\), i.e., the following proposition.
\begin{Proposition}\label{adaptedcodim2sheaf}
  \(\widehat{H}\) is adapted to \(\pF\) in codimension \(2\).
\end{Proposition}
Moreover, we can prove the following proposition by the blow-up technique used in Section~\ref{3}. As it is tautological, we omit the proof.
\begin{Proposition}\label{hatHadapt}
  $\widehat{H}$ is adapted to any parabolic subsheaf $\mathcal{S}_{*}$ of $\pF$ in codimension \(1\). 
\end{Proposition}
Following from Lemma~\ref{curvatureboundbundle}. We also have the curvature bound for $\widehat{H}$.
\begin{Lemma}\label{curvature bound}
  \(|F_{\widehat{H}}|_{\widehat{H}}=O(|\sigma|^{-(1+\varepsilon)})\).
\end{Lemma}
\section{Parabolic stability and Analytic stability}\label{5}
Given a parabolic sheaf $\pF$ over a compact K\"ahler manifold $(X,\omega)$. We define the parabolic degree $\operatorname{deg}_{\omega}(\pF)$ of $\pF$ with respect to $\omega$ by 
\begin{equation*}
  \deg_{\omega}(\pF)\coloneqq \ch_{1}(\pF)\cdot \frac{[\omega]^{n-1}}{(n-1)!},
\end{equation*}
and the slope $\mu_{\omega}(\pF)$ of $\pF$ with respect to $\omega$ by 
\begin{equation*}
  \mu_{\omega}(\pF)\coloneqq \frac{\deg_{\omega}(\pF)}{\rank(\pF)}.
\end{equation*}
\begin{Definition}[Parabolic stability]
  If for any proper parabolic subsheaf $\mathcal{S}_{*}$ of $\pF$, $\mu_{\omega}(\mathcal{S}_{*})<\mu_{\omega}(\pF)$, we say that $\pF$ is parabolic stable with respect to $\omega$. 
\end{Definition}
Next, we recall Simpson's~\cite{Si1} definition of the analytic stability of a Hermitian holomorphic vector bundle $(E, \overline{\partial}, H)$ over a not necessarily compact K\"ahler manifold $(X^{\circ}, \omega)$. The analytic degree $\deg_{H,\omega}(\mathcal{S})$ of a torsion-free subsheaf $\mathcal{S}$ of $E$ is given by 
\begin{equation*}
  \deg_{H,\omega}(\mathcal{S})\coloneqq \frac{\sqrt{-1}}{2\pi}\int_{\mathcal{S}_{reg}}\tr(F_{H_{|\mathcal{S}}})\wedge \frac{\omega^{n-1}}{(n-1)!},
\end{equation*}
where $H_{|\mathcal{S}}$ is the restriction of $H$ to the locally free part of $\mathcal{S}$ and $\mathcal{S}_{reg}$ is the Zariski open subset of $X^{\circ}$ where $\mathcal{S}$ is locally free and the analytic slope $\mu_{H,\omega}(\mathcal{S})$ is defined as 
\begin{equation*}
  \mu_{H,\omega}(\mathcal{S})\coloneqq \frac{\deg_{H,\omega}(\mathcal{S})}{\rank(\mathcal{S})}.
\end{equation*}
Notice that by Chern-Weil's formula we have 
\begin{equation*}
  \deg_{H,\omega}(\mathcal{S})=\frac{\sqrt{-1}}{2\pi}\int_{\mathcal{S}_{reg}}\tr(\operatorname{p}_{\mathcal{S},\widehat{H}}\Lambda _{\omega}F_{H})-\frac{1}{2\pi}\int_{\mathcal{S}_{reg}}\left|\overline{\partial}\operatorname{p}_{\mathcal{S},\widehat{H}} \right|_{H}^{2}
\end{equation*}
where the operator $\Lambda_{\omega}$ stands for the contraction with respect to $\omega$ and $\operatorname{p}_{\mathcal{S}}$ is the orthogonal projection of $E$ onto $\mathcal{S}$. Hence if we want the above notions defined over the non-compact manifold $X^{\circ}$ to make sense, we need to require $F_{H}\in L^{1}(X^{\circ},\omega)$ and $\operatorname{p}_{\mathcal{S},\widehat{H}}\in L_{1}^{2}(S_{H})$ where $L_{1}^{2}(S_{H})$ is the Sobolev space of sections of $\operatorname{End}(E)$ that are self-adjoint with respect to $H$. 

\begin{Definition}[Analytic stability]
  If for any proper saturated subsheaf $\mathcal{S}$ of $E$ such that $\operatorname{p}_{\mathcal{S},\widehat{H}}\in L_{1}^{2}(S_{H})$, $\mu_{H,\omega}(\mathcal{S})<\mu_{H,\omega}(E)$, we say that $E$ is analytic stable with respect to $\omega$ and $H$. 
\end{Definition}
In the last section, we have constructed a metric $\widehat{H}$ that is adpated to any parabolic subsheaf $\mathcal{S}_{*}$ of $\pF$ in codimension $1$. And $\widehat{H}$ is defined on the vector bundle $F_{|X^{\circ}\setminus D}$ where $F$ is the locally free part of the underlying sheaf $\F$. We should have the following proposition. 
\begin{Proposition}
  The parabolic stability of $\pF$ with respect to $\omega$ is equivalent to the analytic stability of $F_{|X^{\circ}\setminus D}$ with respect to the restriction of $\omega$ and $\widehat{H}$.
  \begin{proof}
    Suppose $F$ is analytic stable. For any proper parabolic subsheaf $\mathcal{S}_{*}$ of $\pF$, $\Sb_{|_{X^{\circ}\setminus D}}$ is a torsion-free subsheaf of $F$. Then it follows immediately from Lemma~\ref{hatHadapt} that $\pF$ is parabolic stable.

    Conversely, suppose $\pF$ is parabolic stable, we need to compare the analytic slope of any proper saturated subsheaf $\Sb$ of $F_{|X^{\circ}\setminus D}$ such that $\operatorname{p}_{\mathcal{S},\widehat{H}}\in L_{1}^{2}(S_{\widehat{H}})$ with the analytic slope of $F_{|X^{\circ}\setminus D}$. Again by Lemma~\ref{hatHadapt}, it suffices to show that any saturated subsheaf $\Sb$ of $F$ with $\operatorname{p}_{\mathcal{S},\widehat{H}}\in L_{1}^{2}(S_{\widehat{H}})$ can be extended to a parabolic subsheaf of $\pF$. It was proved in \cite[Proposition 5.9]{Li} that $\Sb$ can be extended to $X^{\circ}$ as a coherent subsheaf of $F$. Since $X-X^{\circ}$ is an analytic subset of codimension $3$ and $\F$ is reflexive, it can be extended further to a coherent subsheaf $\mathcal{S}'$ of $\F$. Let $\mathcal{S}^{\prime\prime}$ be the saturation of $\mathcal{S}'$, then $\mathcal{S}^{\prime\prime}_{*}$ is a parabolic subsheaf of $\pF$ extending $\mathcal{S}$.
   \end{proof}
\end{Proposition}

For the convenience of the readers as well as the later use, we conclude the spirit of~\cite[Proposition 5.9]{Li} in the following proposition. As a preparation, we need the following concept concerning the singularities of a Hermitian metric on a divisor.
\begin{Definition}\label{polyg}
  Let $\Delta^{n}$ be the polydisk certered at origin and $\Delta^{*}$ be the punctured unit disk. Let $E$ be the trivial holomorphic vector bundle over $\Delta^{n}$ and $H$ be a possibly singular metric well defined on $E_{|\Delta^{n-k}\times (\Delta^{*})^{k}}$. We say that $H$ is of polynomial growth along the divisor $z_{k+1}\cdot z_{k+2}\cdots z_{n}=0$, if $|H|=O(|z'|^{a})$ for some $a\in \mathbb{R}$, where $z'\coloneqq (z_{k+1},\cdots,z_{n})$. If $a>0$, we also say that $H$ is of polynomial decay.
\end{Definition}
Then~\cite[Proposition 5.9]{Li} says:
\begin{Proposition}\label{sheafextend}
  If H is an Hermitian metric defined on $F_{X^{\circ}\setminus D}$ that is locally of polynomial growth along $D$, then for any proper saturated coherent subsheaf $\mathcal{S}$ of $F_{X^{\circ}\setminus D}$ such that $\operatorname{p}_{\mathcal{S},\widehat{H}}\in L_{1}^{2}(S_{{H}})$, $\mathcal{S}$ can be extended as a coherent subsheaf of $F$ over $X^{\circ}$. 
\end{Proposition}

\section{Kobayashi-Hitchin Correspondence}\label{6}
We use the Hermitian-Yang-Mills (H-Y-M) flow to deform $\widehat{H}$. We will show that if the parabolic sheaf is stable (resp. semistable), then there is a Hermitian-Einstein (resp. an approximate) strucutre on the bundle $F_{|X^{\circ}\setminus D}$ that is comtatible with the parabolic structure. Since $\widehat{H}$ is only defined on $X^\circ-D$ and we are not clear about the property of $\widehat{H}$ near the singular locus $W$, we investigate the behavior of the H-Y-M flow on $\widetilde{X}\setminus D$ first where $\widehat{H}$ is a well-defined metric on $E_{|\widetilde{X}\setminus D}$ with polynomial decay approaching $D$.
\subsection{Hermitian-Yang-Mills flow}
For any conical K\"ahler metric $\omega_{\epsilon\delta}$ that we have constructed in Section~\ref{4} on $\widetilde{X}\setminus D$, we consider the H-Y-M flow 
\begin{equation}\label{YMflow}
  H^{-1}\frac{\diff H}{\diff t}=-2\left(\sqrt{-1}\Lambda_{\omega_{\epsilon\delta}}F_{H}-\lambda_{\epsilon\delta}\cdot\id_{E'}\right)
\end{equation}
living on $E_{|\widetilde{X}\setminus D}$, where $$\lambda_{\epsilon\delta}\coloneqq\frac{\sqrt{-1}}{\operatorname{Vol}(\widetilde{X},\omega_{\epsilon\delta})}\cdot\int_{\widetilde{X}\setminus D^{*}}\tr(F_{\widehat{H}})\wedge \frac{\omega_{\epsilon\delta}^{n-1}}{(n-1)!}.$$
 The following proposition was obtained by Simpson~\cite{Si1}.
\begin{Proposition}\label{existflow}
  Let $(\widetilde{X}\setminus D,\omega_{\epsilon\delta})$ satisfies the three assumptions mentioned in Proposition~\ref{Simpsonassumption}. Suppose $\widehat{H}$ is a metric satisfies the assumption that $\left\lVert\Lambda_{\omega_{\epsilon\delta}}F_{\widehat{H}}\right\rVert_{L^{\infty}(\widehat{H})}\leq B$ where $B$ is a positive constant. Then there is a unique solution $H(t)$ to the H-Y-M flow with $\det(H)=\det(\widehat{H})$ such that $H(0)=\widehat{H}$, such that $\left\lVert H \right\rVert_{L^{\infty}(\widehat{H})}$ is bounded on each finite interval of time. For this solution, $\left\lVert\Lambda_{\omega_{\epsilon\delta}}F_{H}\right\rVert_{L^{\infty}(H)}\leq B$ for all $t$.
\end{Proposition}
It follows from Lemma~\ref{curvature bound} that, for sufficiently small $\delta$, we have $\left\lVert\Lambda_{\omega_{\epsilon\delta}}F_{\widehat{H}}\right\rVert_{L^{\infty}(\widehat{H})}\leq B$. Hence the above proposition can be applied and we obtain a family of solutions $H_{\epsilon\delta}$ with the same initial value $\widehat{H}$. We wish to show that $H_{\epsilon\delta}$ will converge to a solution of the H-Y-M flow \begin{equation}\label{Ymflow2}
  H^{-1}\frac{\diff H}{\diff t}=-2\left(\sqrt{-1}\Lambda_{\omega}F_{H}-\lambda\cdot\id_{F}\right)
\end{equation}
living on the bundle $F\coloneqq\F_{|X^{\circ}\setminus D}$ where $$\lambda\coloneqq \frac{2\pi\cdot \mu_{\omega}(\pF)}{\operatorname{Vol}(X,\omega)}.$$ Furthermore, under the stability assumption of the parabolic sheaf, we hope that the initial metric $\widehat{H}$ will deform into an Hermitian-Einstein metric with respect to $\omega$ along the flow. The idea of the proof basically comes from~\cite{B-S}. Firstly, we do some estimates.
\begin{Lemma}\label{L1bound}
  $\left\lVert\Lambda_{\epsilon\delta}F_{\widehat{H}}\right\rVert_{L^{1}(\widehat{H},\omega_{\epsilon\delta})}\leq C_{6}$ with $C_{6}$ independent of $\epsilon$ and $\delta$, i.e. $\Lambda_{\epsilon\delta}F_{\widehat{H}}$ is uniformly integrable.
  \begin{proof}
    It follows from Lemma~\ref{curvature bound} that we may fix an $\epsilon_{1}$ and a $\delta_{1}$ such that $\sqrt{-1}\tr(F_{\widehat{H}})\leq C_{7}\cdot\omega_{\epsilon_{1}\delta_{1}}$. Then it holds: 
    \begin{align*}
      \left\lvert\Lambda_{\epsilon\delta}F_{\widehat{H}}\right\rvert\cdot\omega_{\epsilon\delta}^{n}&\leq \left(\left\lvert\Lambda_{\epsilon\delta}(C_{7}\cdot\omega_{\epsilon_{1}\delta_{1}}\cdot I-\sqrt{-1}F_{\widehat{H}})\right\rvert+\left\lvert C_{7}\cdot\Lambda_{\epsilon\delta}\omega_{\epsilon_{1}\delta_{1}}\cdot I\right\rvert\right)\cdot\omega_{\epsilon\delta}^{n}\\
      &\leq n\cdot \tr\left(2\cdot C_{7}\cdot\omega_{\epsilon_{1}\delta_{1}}\cdot I-\sqrt{-1}F_{\widehat{H}}\right)\cdot \omega_{\epsilon\delta}^{n-1}.
    \end{align*}
    Integrating on both sides, we get \begin{align*}
      \int_{\widetilde{X}\setminus D}\left\lvert\Lambda_{\epsilon\delta}F_{\widehat{H}}\right\rvert\cdot\omega_{\epsilon\delta}^{n}&\leq n\cdot \int_{\widetilde{X}\setminus D}(2rC_{7}\cdot \omega_{\epsilon_{1}\delta_{1}}-\sqrt{-1}\tr(F_{\widehat{H}}))\wedge\omega_{\epsilon\delta}^{n-1}\\
      &\leq n\cdot \int_{\widetilde{X}\setminus D}(2rC_{7}\cdot \omega_{1\delta_{1}}-\sqrt{-1}\tr(F_{\widehat{H}}))\wedge\omega_{1\delta}^{n-1}\\
      &=n\cdot \int_{\widetilde{X}\setminus D}(2rC_{7}\cdot \omega_{1}-\sqrt{-1}\tr(F_{\widehat{H}}))\wedge\omega_{1}^{n-1}\\
      &=C_{6}.
    \end{align*}
  \end{proof}
\end{Lemma}
Unless otherwise specified, the constants appear in the estimates in the rest of the paper will always be uniform in $\epsilon$ and $\delta$. 

Along the heat flow~\eqref{YMflow}, we have the following inequalities (c.f.~\cite{Do2}):
\begin{align*}
  &\left(\Delta_{\epsilon\delta}+\frac{\partial}{\partial t}\right)|\Lambda_{\epsilon\delta}F_{H_{\epsilon\delta}}|_{H_{\epsilon\delta}}\leq 0\\
   &\left(\Delta_{\epsilon\delta}+\frac{\partial}{\partial t}\right)|\Lambda_{\epsilon\delta}F_{H_{\epsilon\delta}}|^2_{H_{\epsilon\delta}}\leq 0.
\end{align*}
And we put $f(t)\coloneqq |\Lambda_{\epsilon\delta}F_{H_{\epsilon\delta}}|_{H_{\epsilon\delta}}$.
\begin{Lemma}\label{decreasing}
  $\lVert f_{t} \rVert_{L^{1}}$ and $\lVert f_{t} \rVert_{L^{2}}$ is non-increasing with time.
  \begin{proof}
      Suppose to the contrary that there is $t_{2}>t_{1}$ such that $\lVert f(t_2) \rVert_{L^{1}}=\lVert f(t_1) \rVert_{L^{1}}+\delta$  with $\delta>0$. Since $\lVert f_{t} \rVert_{L^{\infty}}\leq B$ and $|dV_{\omega_{\epsilon\delta}}|< \infty$, we can take a relatively compact region $Z\subset X$ such that $B |dV_{\omega_{\epsilon\delta}}|(X-Z)<\frac{\delta}{8}$. Then we have $$
      \lVert f(t_2) \rVert_{L^{1}(Z)}\geq \lVert f(t_1) \rVert_{L^{1}(Z)}+\frac{3}{4}\delta.
      $$
      
      On the other hand, let $\Omega^{\varphi}$ be a family of nested compact regions with smooth boundaries whose limit exhausts $X^{k}-D^{*}$. Simpson~\cite[Section 6]{Si1} showed that the solution $H_{\epsilon\delta}$ can be obtained by taking the $C^{\infty}_{loc}$ limit of a sequence of metrics $H_{\epsilon\delta}^{\varphi}$ which are the solutions to the H-Y-M flow~\eqref{YMflow} over $\Omega^{\varphi}$ satisfying the Neumann boundary condition. We put $f^{\varphi}(t)\coloneqq |\Lambda_{\epsilon\delta}F_{H_{\epsilon\delta}^{\varphi}}|_{H_{\epsilon\delta}^{\varphi}}$. Then we have 
      $$\left(\Delta_{\epsilon\delta}+\frac{\partial}{\partial t}\right)f^{\varphi}\leq 0.$$
      Integrating by part on both side over $\Omega^{\varphi}$ and using the Neumann boundary condition, we obtain that $\frac{\partial}{\partial t}\lVert f^{\varphi} \rVert_{L^{1}(\Omega^{\varphi})}\leq 0$ for any $t>0$. 
      As $H_{\varphi}(t_{i})$ converges in $C^{\infty}(Z)$ to $H(t_{i})$ where $i=1,2$ as $\varphi\to \infty$,  hence for $\varphi$ sufficiently large, we have \[
      \lvert\lVert f^{\varphi}(t_{i})\rVert_{L^{1}(Z)}-\lVert f(t_{i})\rVert_{L^{1}(Z)} \rvert\leq \frac{\delta}{8}.
      \]
      But \[
      \lVert f(t_2)^{\varphi}\rVert_{L^{1}(X_{\varphi})}\leq \lVert f(t_1)^{\varphi}\rVert_{L^{1}(X_{\varphi})},
      \]
      hence 
      \[
      \lVert f(t_2)^{\varphi}\rVert_{L^{1}(Z)}\leq \lVert f(t_1)^{\varphi}\rVert_{L^{1}(Z)}+\frac{\delta}{4},
      \] 
      therefore 
      \[
      \lVert f(t_2) \rVert_{L^{1}(Z)}\leq \lVert f(t_1) \rVert_{L^{1}(Z)}+\frac{1}{2}\delta,
      \]
      thus a contradiction. The same reasoning applies to $\lVert f_{t} \rVert_{L^{2}}$.
  \end{proof}
\end{Lemma}
As a corollary, we have: 
\begin{Lemma}\label{timeL1bound}
  \begin{equation*}
    \|\Lambda_{\epsilon\delta}F_{H_{\epsilon\delta}}\|_{L^{1}(H_{\epsilon\delta})}\leq C_{6}
  \end{equation*}
  for $t\in [0,\infty[$. 
\end{Lemma}

\begin{Lemma}\label{C0boundmeancurvature}
  For any $t>0$,  $\|\Lambda_{\epsilon\delta}F_{H_{\epsilon\delta}}\|_{L^{\infty}(H_{\epsilon\delta},\omega_{\epsilon\delta})} \leq C_{8}(\max(t^{-1}, 1))$.
  \begin{proof}
    Recall that we have derived a Sobolev inequality for smooth functions compactly supported on the measure space $(\widetilde{X}\setminus D^{*}, \omega_{\epsilon\delta})$ in Lemma~\ref{uniSob} and the Sobolev constant is independent of $\epsilon$ and $\delta$. We may consider the function \[f_{\varsigma}\coloneqq |\Lambda_{\epsilon\delta}F_{H_{\epsilon\delta}}|_{H_{\epsilon\delta}}+\varsigma(\log(|\sigma|^2)-At)\] where $A$ is set as a large number to make sure that $f_{\varsigma\nu}$ is a subsolution to the heat equation. Then $\varphi\coloneqq \eta_a(t)^{2}\cdot (f_{\varsigma\nu}^{+})^{2a-1}$ with $a\geq 1$ is a legitimate test function for the parabolic Moser's iteration technique. Here \(\eta_a(t)\) is an appropriate cutoff function. We obtain for any \(T> 0\),
    \[
    {\lVert f_\varsigma(t)\rVert}_{L^{\infty}((\widetilde{X}\setminus D)\times [1.5T,\,2T])}\leq C_8(T^{-1}){\lVert f_\varsigma(t)\rVert}_{L^1((\widetilde{X}\setminus D)\times [T,\,2T])}.
    \]
    Then the lemma follows by taking the limits on both sides as \(\varsigma\to 0\).
  \end{proof}
\end{Lemma}
We put $B_{\omega_{1}}(d)\coloneqq \{x\in \widetilde{X}\,|\,d_{\omega_{1}}(x, \mathscr{E})< d\}$. Then it follows easily from Lemma~\ref{curvature bound} that $$|\Lambda_{\omega_{\epsilon\delta}}F_{\widehat{H}}|_{\widehat{H}}\leq C_{9}(d^{-1})|\sigma|^{-2}$$ over $B_{\omega_{1}}(d)^{c}$ with $C_{9}(d^{-1})$ independent of \(\epsilon\) and \(\delta\). And we have the following lemma. 
\begin{Lemma}
  For any $t\geq 0$, there exits a constant $C_{10}(d^{-1})$ such that:
  \begin{equation*}
    |\Lambda_{\omega_{\epsilon\delta}}F_{H_{\epsilon\delta}}|_{H_{\epsilon\delta}}\leq C_{10}(d^{-1})|\sigma|^{-2}
  \end{equation*}
  for all $(x,t)\in B_{\omega_{1}}(d)^{c}\times [0, \infty[$.
  \begin{proof}
    The crucial part of the proof is the uniform Gaussian upper bound of the heat kernel $K_{\epsilon\delta}$. But it is obtained in Proposition~\ref{HeatkernelUB}. The rest of the proof follows in the same way as in the proof in~\cite[Lemma 2.2]{Li-Zh-Zh}.
  \end{proof}
\end{Lemma}
We put $h_{\epsilon\delta}\coloneqq H_{\epsilon\delta}\widehat{H}^{-1}$. We know that $|H_{\epsilon\delta}|_{\widehat{H}}$ and $|H_{\epsilon\delta}^{-1}|_{\widehat{H}}$ is comparable to the positive quantity $\varPhi_{\epsilon\delta}=\log(h_{\epsilon\delta})+\log(h_{\epsilon\delta}^{-1})-2\rank(E)$. And we have 
\begin{equation}\label{eqmetricbound}
  \frac{\partial}{\partial t}\varPhi_{\epsilon\delta}\leq 2(|\Lambda_{\omega_{\epsilon\delta}}F_{H_{\epsilon\delta}}|_{H_{\epsilon\delta}}+|\Lambda_{\omega_{\epsilon\delta}}F_{\widehat{H}}|_{\widehat{H}})
\end{equation}
Hence we obtain:
\begin{Lemma}\label{metricbound}
  For $(x, t)\in  B_{\omega_{1}}(d)^{c}\times[0, T]$,
  \begin{align*}
    |H_{\epsilon\delta}|_{\widehat{H}}&\leq TC_{11}(d^{-1})|\sigma|^{-2}\\
    |H_{\epsilon\delta}^{-1}|_{\widehat{H}}&\leq TC_{11}(d^{-1})|\sigma|^{-2}
  \end{align*}
\end{Lemma}
To get the convergence of the H-Y-M flow, we need one more proposition of~\cite[Proposition 1]{B-S}.
\begin{Proposition}
  Let $H$ be an Hermitian matrix valued function defined on a K\"ahler manifold $(Y,\omega)$ which belongs to the Sobolev space $L_1^{2}$. Assume that $H$ and $H^{-1}$ is uniformly bounded and it satisfies the elliptic equation 
  \begin{equation*}
    \Lambda_{\omega}\overline{\partial}(\partial HH^{-1})=f
  \end{equation*} 
  in a weak sense with a uniformly bounded function $f$. Then $H$ belongs to $C^{1,\alpha}_{loc}$ for any $0<\alpha<1$ and admits an estimate depending only on $\|H\|_{L^{\infty}}$, $\|H^{-1}\|_{L^{\infty}}$, $\|f\|_{L^{\infty}}$ and the geometry of $Y$. 
\end{Proposition}
Now we can apply the above proposition to do the interior estimates for $H_{\epsilon\delta}$ over $(\widetilde{X}^{\circ}-D,\,\omega_{\epsilon\delta})$. Notice that $\omega_{\epsilon\delta}$ is locally uniformly quasi-isometric to the fixed K\"ahler metric $\omega_{1}$. Then we simply set $\Lambda_{\omega_{\epsilon\delta}}F_{H_{\epsilon\delta}}$ as the $f$ in the above proposition. Hence $H_{\epsilon\delta}$ is uniformly bounded in $C^{1,\alpha}_{loc}$-topology for any $t\geq 0$. On the other hand, by the H-Y-M flow, we can see that $\frac{\diff H_{\epsilon\delta}}{\diff t}$ is uniformly bounded in  $C^{0}_{loc}$-topology for any $t > 0$. Hence $H_{\epsilon\delta}$ converges in $C^{1/0}_{loc}$-topology to an time flow of Hermitian metric $H$ living on $E_{|\widetilde{X}\setminus D^{*}-\pi^{-1}(W)}$, or equivalently on $F\coloneqq \F_{X\setminus (D\cup W)}$ with the initial value $\widehat{H}$. And we may apply the parabolic Schauder estimate to show that $H$ is indeed a smooth solution and $H_{\epsilon\delta}$ converges in $C^{\infty/\infty}_{loc}$-topology to $H$. Moreover, due to Lemma~\eqref{eqmetricbound}, for any $t>0$, we have that $H(t)$ is locally of polynomial decay approaching $D$. 

\subsection{Correspondence}
\paragraph{From stablity to H-E metric}
Based on the H-Y-M flow $H(t)$, we want to show that under the stability condition of the parabolic sheaf $\pF$, $H$ will converge to an H-E metric on $X^{\circ}\setminus D$ which is compatible with the parabolic structure. Parallelly, we also want to show that under the semistability condition, $H(t)$ will give us a family of approximate H-E metrics all of which are compatible with the parabolic structure. 

Indeed, the difficulty comes from the fact that $|\Lambda F_{\widehat{H}}|_{\widehat{H}}$ is not bounded on $X^{\circ}\setminus D$. Hence Simpson's~\cite{Si1} arguments cannot be directly applied. One may argue that one can consider the H-Y-M flow starting from a positive time point and then apply Simpson's results. Although from the above analyses, we know that \(|\Lambda F_{H(t)}|_{H(t)}\) is bounded for any $t>0$, but it seems difficult to show the analytic stability of \(F_{|X^{\circ}\setminus D}\) with respect to $H(t)$. But this dilemma was resolved by the second author in~\cite[Proposition 4.1]{Li-Zh-Zh}, which shows that under the semistability condition, we still have \[\lim_{t\to \infty}\lVert\varPhi(t)\rVert_{L^{2}(H(t))}=0\]
where \(\varPhi(t)\coloneqq\lVert \sqrt{-1}\cdot\Lambda F_{H(t)}-\lambda\cdot\id_{F}  \rVert_{L^{2}(H(t))}.\)
This implies the existence of approximate H-E metrics. Indeed, we have 
\[
\left(\frac{\partial}{\partial t}+\Delta_{\omega}\right)\cdot\varPhi(t)^2\leq 0.
\]
Then we choose a test function for Moser's iteration as in~Lemma~\ref{C0boundmeancurvature} to get
\[
\lVert \varPhi(t)  \rVert_{L^{\infty}(H(t))(X\setminus D\times [1.5T,\,2T])}\leq C_{12}\lVert \varPhi(t) \rVert_{L^{2}(H(t))(X\setminus D\times [T,\,2T])}.
\]

For the stable case, the same trick used in~\cite[Proposition 4.1]{Li-Zh-Zh} can be applied to show that Proposition 5.3 of~\cite{Si1} holds for $H(t)$ which is the crucial estimate calling for the stability condition. Then the arguments in Section 7 of~\cite{Si1} can be applied to show that $H(t)$ converges to an H-E metric $H(\infty)$.  

Hence we only need to show that the metrics (H-E or approximate H-E) are compatible with the parabolic structure and admissible in the sense of Definition~\ref{comatibleandadmissible}. It suffices to prove for the approximate H-E metrics as the H-E case follows directly from Fatou's lemma. 

We need a lemma first. 
\begin{Lemma}[Lemma 5.2 in~\cite{Si1}] \label{vanishinglemma}
 Suppose $Y$ is a noncompact K\"ahler manifold which has an exhaustion function $\phi$ with $\int_{Y}|\Delta \phi|< \infty$, and suppose $\eta$ is a $(2n-1)$-form with $\int_{Y}|\eta|^{2}<\infty$. Then if $\diff\eta$ is integrable, $\int_{Y}\diff \eta=0$.
\end{Lemma}
Let us fix a K\"ahler metric $\omega_c$ on \((\widetilde{X}\setminus D)\) with cusp singularities along $D$. It follows that the density functions $\frac{\omega_{\epsilon\delta}^n}{\omega_c^n}$ are uniformly bounded in $\epsilon$ and $\delta$. 
\begin{Proposition}
  If $\pF$ is semistable, then for any $t>0$, $H(t)$ is compatible with the parabolic structure. 
\end{Proposition}
\begin{proof}
    As \(\det(\widehat{H})=\det(H(t))\), \(H(t)\) is adapted to \(\pF\) in codimension \(1\). Hence it suffices to show that
    
     \[\sqrt{-1}\cdot\int_{X^{\circ}\setminus D}\tr(F_{\widehat{H}_{|\mathcal{S}}})\wedge\omega^{n-1}\leq\sqrt{-1}\cdot\int_{X^{\circ}\setminus D}\tr(F_{H(t)_{|\mathcal{S}}})\wedge\omega^{n-1}\] 
     for any proper parabolic subsheaf \(\pS\).
     
     By the blow-up procedure in Section~\ref{3}, it can be shown that 
    \begin{equation*}
      \int_{\widetilde{X}^{\circ}\setminus D^{*}}\tr(F_{\widehat{H}_{|\mathcal{S}}})\wedge\omega^{n-1}=\int_{X^{\circ}\setminus D}\tr(F_{\widehat{H}_{|\mathcal{S}}})\wedge\omega^{n-1}.
    \end{equation*}

    An adaptation of Lemma~\ref{L1bound} yields the following bound:
    \begin{equation*}
      \left\lvert\Lambda_{\epsilon\delta}F_{\widehat{H}_{|\mathcal{S}}}\right\rvert\cdot\omega_{\epsilon\delta}^{n}\leq n\cdot \tr\left(2\cdot C_{7}\cdot\omega_{1\delta_{1}}\cdot I-\sqrt{-1}F_{\widehat{H}_{|\mathcal{S}}}\right)\cdot \omega_{1\delta}^{n-1}.
    \end{equation*}
    Since the integral of the right-hand side is finite and independent of $\delta$, then by generalized dominated convergence theorem, we obtain 
    \begin{equation*}
      \int_{X^{\circ}\setminus D}\tr(F_{\widehat{H}_{|\mathcal{S}}})\wedge\omega^{n-1}=\lim_{\substack{\epsilon\to 0 \\\delta\to 0}}\int_{\widetilde{X}^{\circ}\setminus D^{*}}\tr(F_{\widehat{H}_{|\mathcal{S}}})\wedge\omega_{\epsilon\delta}^{n-1}.
    \end{equation*}
    
     For fixed \(\epsilon\), \(\delta\) and \(t\), we put $h_{\epsilon\delta|\mathcal{S}}\coloneqq H_{\epsilon\delta}(t)_{|\mathcal{S}}\cdot\widehat{H}_{|\mathcal{S}}^{-1}$ and $h_{\epsilon\delta}\coloneqq H_{\epsilon\delta}(t)\cdot\widehat{H}^{-1}$. Simpson~\cite{Si1} showed that \(\widehat{H}\) and \(H_{\epsilon\delta}(t)\) are mutually bounded and \(\overline{\partial}h_{\epsilon\delta}\in L^2(\widehat{H},\omega_{\epsilon\delta})\). Hence \(h_{\epsilon\delta}\) is bounded with respect to either \(\widehat{H}\) or \(H_{\epsilon\delta}(t)\).

    Since we have \begin{align}
      {h_{\epsilon\delta}}_{|\mathcal{S}}&=\operatorname{p}_{\mathcal{S}, \widehat{H}}\cdot h_{\epsilon\delta}\cdot \iota \label{eq:first}\\
      \operatorname{p}_{\mathcal{S},H_{\epsilon\delta}(t)}&={h_{\epsilon\delta}}_{|\mathcal{S}}^{-1}\cdot \operatorname{p}_{\mathcal{S},\widehat{H}}\cdot h_{\epsilon\delta}\label{eq:second}
    \end{align}
    \[\] where $\iota$ is the injection of $\mathcal{S}$ into $E$. Hence, \[\overline{\partial}{h_{\epsilon\delta}}_{|\mathcal{S}}=\overline{\partial}\operatorname{p}_{\mathcal{S}, \widehat{H}}\cdot h_{\epsilon\delta}\cdot \iota+\operatorname{p}_{\mathcal{S}, \widehat{H}}\cdot \overline{\partial}h_{\epsilon\delta}\cdot \iota\] which implies that $\overline{\partial}h_{\epsilon\delta|\mathcal{S}}\in L^{2}(\widehat{H},\omega_{\epsilon\delta})$. 
    Then by Chern-Weil's formula and \eqref{eq:second}, 
    \[\tr(F_{{H_{\epsilon\delta}}_{|\mathcal{S}}})\wedge \omega_{\epsilon\delta}^{n-1}\]
    is integrable. 
    
    On the other hand, \[\tr(F_{{H_{\epsilon\delta}}_{|\mathcal{S}}})-\tr(F_{\widehat{H}_{|\mathcal{S}}})=\tr(\overline{\partial}(\partial_{\widehat{H}_{|\mathcal{S}}}{h_{\epsilon\delta}}_{|\mathcal{S}}\cdot {h_{\epsilon\delta}}_{|\mathcal{S}}^{-1})).\]
    But the above reasoning shows that
    \[
    \tr(\partial_{\widehat{H}_{|\mathcal{S}}}{h_{\epsilon\delta}}_{|\mathcal{S}}\cdot {h_{\epsilon\delta}}_{|\mathcal{S}}^{-1})\wedge \omega_{\epsilon\delta}^{n-1}\in L^2(\widehat{H},\omega_{\epsilon\delta}).
    \]
    Then it follows from Lemma~\ref{vanishinglemma} that 
    \begin{equation*}
      \int_{\widetilde{X}^{\circ}\setminus D^{*}}\tr(F_{\widehat{H}_{|\mathcal{S}}})\wedge\omega_{\epsilon\delta}^{n-1}=\int_{\widetilde{X}^{\circ}\setminus D^{*}}\tr(F_{{H_{\epsilon\delta}}_{|\mathcal{S}}})\wedge\omega_{\epsilon\delta}^{n-1}.
    \end{equation*}

    By Chern-Weil's formula again, we have
    \begin{equation*}
      \sqrt{-1}\cdot\tr(F_{{H_{\epsilon\delta}}_{|\mathcal{S}}})\wedge\omega_{\epsilon\delta}^{n-1}\leq \sqrt{-1}\tr(\operatorname{p}_{\mathcal{S},H_{\epsilon\delta}}\Lambda_{\epsilon\delta}F_{{H_{\epsilon\delta}}})\cdot \frac{\omega_{\epsilon\delta}^{n}}{\omega_c^n}\cdot \omega_c^n.
    \end{equation*}
    But the term on the right-hand side is uniformly bounded as both of $|\Lambda_{\epsilon\delta}F_{{H_{\epsilon\delta}}}|_{H_{\epsilon\delta}}$ and $\frac{\omega_{\epsilon\delta}^n}{\omega_c^n}$ are uniformly bounded. Hence if we take the limit on both sides and apply Fatou's lemma, we obtain the desired inequality: 
    \begin{equation*}
      \sqrt{-1}\cdot\int_{X^{\circ}\setminus D}\tr(F_{\widehat{H}_{|\mathcal{S}}})\wedge\omega^{n-1}\leq\sqrt{-1}\cdot\int_{X^{\circ}\setminus D}\tr(F_{H(t)_{|\mathcal{S}}})\wedge\omega^{n-1}.
    \end{equation*} 
  \end{proof}
  \begin{Proposition}
    For any \(t>0\), \(H(t)\) is admissible. 
  \end{Proposition}
\begin{proof}
    By Lemma~\ref{C0boundmeancurvature}, $|\Lambda F_{H}|_{H}$ is bounded on $X\setminus D$ for any $t>0$. 
    And we have the following:
    \begin{equation}\label{Fatou}
      \begin{aligned}
        &\quad -8\pi^2 \int_X \ch_2(\pF) \wedge \frac{\omega^{n-2}}{(n-2)!} \\
        &= \lim_{\substack{\epsilon \to 0\\ \delta\to 0}} \int_{\widetilde{X}} \operatorname{tr} \left( F_{\widehat{H}} \wedge F_{\widehat{H}} \right) \wedge \frac{\omega_{\epsilon\delta}^{n-2}}{(n-2)!} \\   
        &\geq \lim_{\substack{\epsilon \to 0\\ \delta\to 0}} \int_{\widetilde{X}} \operatorname{tr} \left( F_{  H_{\epsilon\delta}(t)} \wedge F_{  H_{\epsilon\delta}(t)} \right) \wedge \frac{\omega_{\epsilon\delta}^{n-2}}{(n-2)!} \\  
        &=\lim_{\substack{\epsilon \to 0\\ \delta\to 0}} \int_{\widetilde{X}} \left( \left| F_{H_{\epsilon\delta}(t)} \right|^2_{H_{\epsilon\delta}(t),\omega_{\epsilon\delta}}
        - \left| \Lambda_{\epsilon\delta} F_{H_{\epsilon\delta}(t)} \right|^2_{H_{\epsilon\delta}(t)} \right) \cdot\frac{\omega_{\epsilon\delta}^n}{\omega_c^n} \cdot \frac{\omega_c^n}{n!}\\ 
        &\geq \int_{X} \left( \left| F_{H(t)} \right|^2_{H(t),\omega}
        - \left| \sqrt{-1} \Lambda_\omega F_{H(t)} \right|^2_{H(t)} \right) \cdot\frac{\omega^n}{n!}
        \end{aligned}
    \end{equation} 
  where the first inequality follows from \cite[Lemma 7.1]{Li} and the second inequality follows by applying Fatou's lemma with respect to the measure induced by $\omega_c$. Hence for any \(t>0\), \(F_{H(t)}\) is square integrable.
  \end{proof}

\paragraph{From H-E metric to stability}
So far, we have done one direction of the Kobayashi-Hitchin correspondence. 

To prove the converse part, suppose $H(t)$ is a family of approximate H-E metrics compatible with the parabolic structure, then by definition and Chern-Weil's formula, for any proper subsheaf \(\pS\), we have 
\[
\begin{aligned}
  \mu_{\omega}(\pS)\leq &\liminf_{t\to \infty}\frac{\sqrt{-1}}{2\pi\rank(\mathcal{S})}\int_{X\setminus D} \tr(F_{H_{|\mathcal{S}}(t)})\wedge \omega^{n-1}\\
  &\leq \lim_{t\to\infty}\frac{\sqrt{-1}}{2\pi\rank(\mathcal{\F})}\int_{X\setminus D} \tr(F_{H(t)})\wedge \omega^{n-1}\\
  &=\mu_{\omega}(\pF).
\end{aligned}
\]
hence we are done with the semistable case.

To prove the converse part for the polystable case, we need a proposition first, which is a slight generalization of~\cite[Theorem 2 b)]{B-S}. 
\begin{Proposition}\label{regularitynearsingular}
  Let \((\Delta^{n},\,z_1,\ldots,z_n)\) be a polydisk and $D$ be a divisor defined by $z_{1}\cdot z_{2}\cdots z_{m}=0$. Let $S$ be a closed subset with locally finite Hausdorff measure of real codimension $4$. Let $\pF$ be a saturated reflexive parabolic sheaf on $\Delta^{n}$ whose regular part \(F\) is defined on \(\Delta^n\setminus (D\cap S)\). Suppose \(\pF\) admits an admissible H-E metric \(H\) on \(F\) which is compatible with the parabolic structure. Then for any local section $s$ of $\mathcal{F}$, $H(s,s)$ is locally bounded and it belongs to \(L_{1,\,loc}^2\).   
\end{Proposition}
\begin{proof}
    The proof basically follows from~\cite[Section 1]{B-S}. 
    
    For a projection $p$ from $\Delta^{n}$ to $\Delta^{n-2}$ along a generic direction, the set $S\cap p^{-1}(0)$ consists of a countable number of points which may accumulate only at $0$. And there is a compact subset $K$ of $\Delta^{2}$ such that $S$ is contained in $K\times \Delta^{n-2}$. 
    
    We put $X_{t}\coloneqq p^{-1}(t)$. Then except for $t$ in a subset of measure zero of \(\Delta^{n-2}\), we have
    \begin{itemize}
      \item \(S_{t}\coloneqq X_{t}\cap S\) contains only finite points.
      \item \(D_{t}\coloneqq D\cap p^{-1}(t)\) is a simple normal crossing divisor. 
    \end{itemize}
    
    We put $u\coloneqq \log^{+}(H(s,s))$ and $u_{t}\coloneqq u_{|X_{t}}$. We analyse \(u_t\) within each slice \(X_t\). 
    
    Suppose \(x_0\) is an isolated point of \(S_t\) away from \(D_t\). Then it is proved in~\cite[Section 3]{Ba} that \(u_t\) belongs to \(L_1^2\) near \(x_0\) and satisfies the following inequality weakly
    \[
    \Delta_{t}u_{t}\leq 4|F_{t}|.
    \]
    
    If \(x_0\) is contained in \(D_t\), then without loss of generality, we may choose a local coordinate neighborhood \((U_{x_0},\,z_1,z_2)\) in \(X_t\) centered at \(x_0\) such that \(D_t\) is defined by \(z_1\cdot z_2=0\). Due to the fact that \(\pF\) is a saturated reflexive parabolic sheaf and the slice \(U_{x_0}\) is of complex dimension \(2\), we know that \({\pF}_{|U_{x_0}}\) is a parabolic bundle. 
    
    On the other hand, \(F_{|U_{x_0}\setminus D_t}\) admits an admissible H-E metric \(H_t\). A regularity theorem in real dimension \(4\) with singularities in real codimension \(2\) (cf. \cite{Si-Si}) from gauge theory implies that if the curvature tensor \(F_{H_t}\) of a H-E metric belongs to \(L^2\), then it belongs to \(L^p\) for some \(p>2\). Then O.Biquard~\cite[Theorem 2.1]{Bi2} proved that \(F_{|U_{x_0}\setminus D_t}\) can be uniquely extended as a parabolic bundle \(F_*\) to \(U_{x_0}\) by choosing an appropriate complex gauge. In particular, for any holomorphic section \(s\) of \(F_*\), \(|s|_{H_t}\) is bounded. By the uniqueness of the extension, we know that \(F_*\cong{\pF}_{|U_{x_0}}\). Hence, in this case, we also have the weak inequality
    \[
    \Delta_{t}u_{t}\leq 4|F_{t}|.
    \]
  \begin{Remark}
    Biquard was dealing with the extension problem in the case of a smooth divisor. The choice of a complex gauge involves solving a \(\overline{\partial}\)-problem on a punctured polydisk \(\Delta^*\times \Delta\). But the proof could be easily adapted to the case of a simple normal crossing divisor, i.e., solving the same \(\overline{\partial}\)-problem on \(\Delta^*\times \Delta^*\). 
  \end{Remark}
  The remainder of the proof proceeds verbatim as in~\cite[Section 1]{B-S}. Indeed, for any compact region $K^{\prime}\subset \Delta^{n-2}$ containing $0$, we can use the above inequality to show that $\nabla_{t}u$ is square integrable over $K\times K^{\prime}$ where $\nabla_{t}$ is the gradient in the direction of the projection. As the direction of projection is generic, we have $u\in L^{2}_{1,loc}(\Delta^n)$. Once this is known, it is easy to see that $u\in L^{\infty}_{loc}$ as $\Lambda F_{H}\in L^{\infty}_{loc}$ in view of 
    \begin{equation*}
      \Delta u\leq 2|\Lambda F_{H}|.
    \end{equation*}
  \end{proof}
Now suppose $\pF$ is a saturated reflexive parabolic sheaf that admits an admissible H-E metric $H$ compatible with the parabolic structure. Then it follows easily from the definition of \(H\) and Chern-Weil's formula that for any proper parabolic subsheaf $\pS$, we have \(
\mu_{\omega}(\pS)\leq \mu_{\omega}(\pF)
\).

Suppose the equality holds. We put $\mathcal{G}_{*}\coloneqq\wedge^{k}\pF\otimes\det(\mathcal{S}_{*})^{-1}$. As $\det(\mathcal{S}_{*})^{-1}$ is a parabolic line bundle, there exists an admissible H-E metric $H^{\prime}$ compatible with its parabolic structure. Hence the regular part of $\mathcal{G}$ inherits an admissible H-E metric $H^{\prime\prime}$ from $H$ and $H^{\prime}$ 
\begin{equation*}
  \lambda(\mathcal{G})=\frac{2k\pi}{\operatorname{Vol}(X,\omega)}(\mu_{\omega}(\pF)-\mu_{\omega}(\mathcal{S}_{*}))=0.
\end{equation*}

On the other hand, by the above proposition, we know that \(|s|_H^{\prime\prime}\) is bounded on \(X\) and it belongs to \(L^2_{1}(X)\). 
Therefore, we have 
\begin{equation*}
  \Delta_{\omega}|s|_{H^{\prime\prime}}^{2}\leq 2\lambda(\mathcal{G})|s|_{H^{\prime\prime}}^{2}-2|\overline{\partial}s|_{H^{\prime\prime}}^{2}\leq 0
\end{equation*}
which implies $s\equiv C\neq 0$. Hence we have the holomorphic splitting: 
\[{\F}_{|X^{\circ}}={\mathcal{S}}_{|X^{\circ}}\oplus \mathcal{Q}_{|X^{\circ}}\]
where \(\mathcal{Q}\) is subsheaf of \(\F\). 

Since \(W\) has codimension at least \(3\), we have \[\operatorname{Ext}_1(\mathcal{S},\mathcal{Q})\cong \operatorname{Ext}_1(\mathcal{S}_{|X^{\circ}},\mathcal{Q}_{|X^{\circ}}).\] 
Hence we have \[\pF=\pS\oplus \mathcal{Q}_*.\] 

As a consequence, we have established the Kobayashi-Hitchin correspondence:
\begin{Theorem}\label{Stable}
  A saturated reflexive parabolic sheaf $\pF$ over $(X,\omega,D)$ is $\mu_{\omega}$-polystable if and only if there exists an admissible Hermitian-Einstein metric with respect to $\omega$ on ${F_{*}}_{|X^{\circ}\setminus D}$ which is compatible with $\pF$. 
\end{Theorem}
\begin{Theorem}\label{Semistable}
  A saturated reflexive parabolic sheaf $\pF$ over $(X,\omega,D)$ is $\mu_{\omega}$-semistable if and only if there exists a family of approximate Hermitian-Einstein metrics with respect to $\omega$ on ${F_{*}}_{|X^{\circ}\setminus D}$ all of which are compatible with $\pF$. 
\end{Theorem}
\begin{Corollary}\label{BGnormal}
  If $\pF$ is $\mu_{\omega}$-semistable with respcet to a $\omega$, then $$\Delta(\pF)\cdot [\omega]^{n-2}\geq 0.$$ Moreover, if $\pF$ is polystable, then the equality holds if and only if ${\F}_{|X\setminus D}$ is a vector bundle with a projectively flat Hermitian-Einstein connection compatible with the parabolic structure of $\pF$.
\end{Corollary} 
\begin{proof}
  An obvious modification to the calculations in~\eqref{Fatou} yields the proof.
\end{proof}
\section{Bogomolov-Gieseker inequality for nef and big class}\label{7}
In this section, we prove a more general Bogomolov-Gieseker type inequality, i.e. the following theorem.
\begin{Theorem}\label{nefbigBG}
  Let $\pF$ be a saturated reflexive parabolic sheaf over a compact K\"ahler manifold $(X, \omega)$ which is semistable with respect to a nef and big class $[\eta]$. Then the Bogomolov-Gieseker inequality with respect to $\eta$ holds, i.e.
  \begin{equation*}
    \Delta(\pF)\cdot [\eta]^{n-2}\geq 0.
  \end{equation*}
\end{Theorem}
Let us briefly recall the definitions of nefness and bigness. 
\begin{Definition}\label{nef}
  A class $[\eta]\in H^{k,k}(X,\mathbb{R})$ is called nef if for any $\epsilon>0$, there exists a representative $\eta_{\epsilon}\in [\eta]$ such that $\eta_{\epsilon}\geq -\epsilon\omega^{k}$. It is clear that these classes form a closed cone in $H^{k,k}(X,\mathbb{R})$ which we denote as $\mathcal{N}^{k}$.
\end{Definition}
\begin{Definition}\label{big}
  A class $[\eta]\in H^{k,k}(X,\mathbb{R})$ is called big if there exists a constant $\epsilon$ such that $\eta\geq \epsilon [\omega^{k}]$ in the sense of current. It is clear that these classes form an open cone in $H^{k,k}(X,\mathbb{R})$ which we denote as $\mathcal{B}^{k}$.
\end{Definition}
We put $[\eta_{\epsilon}]\coloneqq[\eta+\epsilon\omega]$ where \(\omega\) is a K\"ahler class. We have the following lemma concerning the openness of stablity of $\pF$ with respect to $[\eta_{\epsilon}]$.
\begin{Lemma}\label{openstable}
  Suppose a parabolic sheaf $\pF$ (not necessarily reflexive) is stable with respect to a class $[\eta]\in\mathcal{B}^{1}$. Then $\pF$ is stable with respect to $[\eta_{\epsilon}]$ for any sufficiently small $\epsilon$.
 
\end{Lemma}
 \begin{proof}
    Let $\pS$ be a proper parabolic subsheaf of $\pF$. We have 
    \begin{equation*}
      \mu_{\eta_{\epsilon}}(\pF)-\mu_{\eta_{\epsilon}}(\pS)=\mu_{\eta}(\pF)-\mu_{\eta}(\pS)+\epsilon\cdot\phi(\epsilon).
    \end{equation*} 

    If we can show that 
    \begin{enumerate}
      \item $|\phi(\epsilon)|\leq C_{12}$,
      \item $\mu_{\eta}(\pF)-\mu_{\eta}(\pS)\geq C_{13}$
    \end{enumerate}
     with $C_{12}$, $C_{13}$ independent of the choices of $\pS$ and \(\epsilon\), we are done. 
    
    We show the uniform bound of $\phi(\epsilon)$ first. As $\ch_{1}(\pS)$ is a linear combination of $\ch_{1}(\mathcal{S})$ and $D_{i}$'s, hence $\phi(\epsilon)$ is a linear combination of the intersection numbers of $\ch_{1}(\F)$, $\ch_{1}(\mathcal{S})$ and $D_{i}$'s with some bounded big classes parametrized by $\epsilon$. Thus it suffices to show the following assertion:
    
      Given a family of bounded big classes $\gamma_{\epsilon}$, $\ch_{1}(\mathcal{S})\cdot [\gamma_{\epsilon}]$ has an uniform upper bound if we vary $\mathcal{S}$ and $\epsilon$. 
      
      Following the spirit of Section~\ref{4}, we can contructed a metric $H$ on the regular part of a torsion-free sheaf $\F$ which is adapted to any torsion free subsheaf $\mathcal{S}$ in codimension $1$ (although we were dealing with a parabolic sheaf there, the idea of course could be transplanted to the ordinary sheaf case). Then by Gauss-Codazzi formula, we have \[\ch_{1}(\mathcal{S})=\frac{\sqrt{-1}}{2\pi}\tr(\operatorname{p}_{\mathcal{S},H_0}\cdot F_{H_{0}})-\tr((\overline{\partial}\operatorname{p}_{\mathcal{S},H_0})^{\dagger}\wedge\overline{\partial}\operatorname{p}_{\mathcal{S},H_0})\] in the sense of current. Then it is not difficult to see that $\ch_{1}(\mathcal{S})\cdot [\gamma_{\epsilon}]$ is uniformly bounded. 

      The proof will be completed by the following lemma which implies \(2\).
  \end{proof}
\begin{Lemma}\label{uniformgap}
  Let $\pF$ be a parabolic sheaf and $[\eta]\in \mathcal{B}^{1}$. We put $\mu_{\eta}^{\star}\coloneqq\sup\{\mu_{\eta}(\pS)| \pS \subsetneq\pF \}$. Then $\mu_{\eta}^{\star}$ can be achieved by some proper subsheaf $\pS^{\star}$.
  \begin{proof}
    It follows from~\cite[Lemma 2.2]{Ca} that we can express $\eta^{n-1}=\sum_{i=1}^{s}\lambda_{i}\cdot e_{i}$ with $\lambda_{i}\geq0$, $e_{i}\in H^{2(n-1)}(X,\mathbb{Q})$ and each $e_{i}$ can be represented by a strictly positive current (not necessarily an (n-1,n-1)-current). Then we have 
    \begin{equation*}
      \ch_{1}(\pS)\cdot[\eta^{n-1}]=\sum_{i=1}^{s}\lambda_{i}\cdot \ch_{1}(\pS)\cdot[e_{i}].
    \end{equation*}
    Since we want to achieve the maximal slope, we may restrict ourselves to the set of subsheaves with $-C_{14}<\deg_{\eta}(\pS)$ with $C_{14}$ a positive constant. By the arguments in the above lemma, we may choose the $C_{14}$ large enough such that $\ch_{1}(\pS)\cdot[\eta^{n-1}]<C_{14}$ for all subsheaves. We put $\mathcal{A}\coloneqq\{\pS\subsetneq\pF,\pS\neq0\,|\,-C_{14}<\ch_{1}(\pS)\cdot[\eta^{n-1}]\}$. Then for any $\pS\in \mathcal{A}$ and index $i$ with $\lambda_{i}\neq 0$ we have $-C_{15}<\ch_{1}(\pS)\cdot[e_{i}]<C_{15}$.  
    Recall that 
    \begin{equation*}
      \ch_1(\pS)=\ch_{1}(\mathcal{S})+\sum_{i\in I}\sum_{a\in [0, 1[}a\cdot\rank_{D_{i}}({}^{i}\Gr_{a}\pS)\cdot[D_{i}].
    \end{equation*}
    Since $\mathcal{A}$ is a collection of  parabolic subsheaves, as $\pS$ varies in $\mathcal{A}$, the nontrivial parabolic weights $a$ in the above expression take values from a finite set of real numbers.
    As $\ch_{1}(\mathcal{S})$ and $[D_{i}]$ belong to $H^{2}(X,\mathbb{Z})$ and $e_{i}\in H^{2(n-1)}(X,\mathbb{Q})$, we see that $\ch_{1}(\pS)\cdot e_{i}$ can only achieve finite number of values as we vary $\pS\in\mathcal{A}$ and the index $i$ whence so does the values of $\ch_{1}(\pS)\cdot[\eta^{n-1}]$. The proof is completed. 
  \end{proof}
\end{Lemma}

Let $\pF$ be a saturated reflexive parabolic sheaf which is semistable with respect to a big class $[\eta]\in \mathcal{B}^{1}$. It can be proved as in the ordinary sheaf case that there exists a Jordan-Hölder filtration 
\[
0 = {\mathcal{F}_0}_{*} \subset {\mathcal{F}_1}_{*} \subset {\mathcal{F}_2}_{*} \subset \cdots \subset {\mathcal{F}_n}_{*} = \pF
\]
where for each $0\leq i<n$, ${\mathcal{F}_i}_*$ is a saturated reflexive parabolic sheaf, $\Gr_{i}\pF\coloneqq{\mathcal{F}_{i+1}}_{*}/{\mathcal{F}_i}_{*}$ is a $\eta$-stable parabolic sheaf with $\mu_{\eta}(\Gr_{i}\pF)=\mu_{\eta}(\pF)$.

Then it follows directly from Lemma~\ref{openstable} that there exists an $\epsilon_{0}$ such that for any $0<\epsilon<\epsilon_{0}$ the quotient ${\mathcal{F}_{i+1}}_{*}/{\mathcal{F}_i}_{*}$ is $\eta_{\epsilon}$-stable. 
\begin{proof}[Proof of Theorem~\ref{nefbigBG}]
   ${\mathcal{F}_0}_*$ is a saturated reflexive parabolic sheaf that is stable with respect to $\eta_{\epsilon}$. Under the assumption that $\eta$ is nef and big, we know that $\eta_{\epsilon}$ is a K\"ahler metric. Then we have $\Delta({\mathcal{F}_0}_*)\cdot [\eta_{\epsilon}^{n-1}]\geq 0$ whence $\Delta({\mathcal{F}_0}_*)\cdot [\eta^{n-1}]\geq 0$ by taking limit. On the other hand, Lemma~\ref{stableofreflexivacation} implies that the reflexive saturation $\Gr_{i}\pF'$ of $\Gr_{i}\pF$ is $\eta_{\epsilon}$-stable. Then it follows from Lemma~\ref{BGreflexivication} that $$\Delta(\Gr_{i}\pF)\cdot [\eta^{n-1}]\geq \Delta(\Gr_{i}\pF')\cdot [\eta^{n-1}]\geq 0.$$ Hence it suffices to show that if we have the short exact sequence of parabolic sheaves 
\[
\begin{tikzcd}
  0 \arrow[r] & {\mathcal{F}_i}_{*} \arrow[r] & {\mathcal{F}_{i+1}}_{*} \arrow[r] & \Gr_{i}\pF \arrow[r] & 0
\end{tikzcd},
\]
with $\Delta({\mathcal{F}_i}_{*})\cdot [\eta^{n-1}]\geq 0$ and $\Delta(\Gr_{i}\pF)\cdot [\eta^{n-1}]\geq 0$, we can obtain that $\Delta({\mathcal{F}_{i+1}}_{*})\cdot [\eta^{n-1}]\geq 0$. But this is a standard fact, cf. \cite[Lemma 3.7]{Cl}.
\end{proof}

\printbibliography
\end{document}